\documentclass[12pt, a4paper]{amsart}
\usepackage{a4wide}

\usepackage{hyperref}
\usepackage{lineno}
\usepackage{enumerate}
 \numberwithin{equation}{section}
 \usepackage[shortlabels]{enumitem}

\usepackage{amsmath}
\usepackage{amssymb, amsfonts,amscd,verbatim,hyperref,cleveref,graphics}
\usepackage{xspace}
\usepackage{tikz}
\usetikzlibrary{matrix,calc,decorations.pathmorphing,shapes,arrows}

\newtheorem*{product}{Product Semigroups}
\newtheorem*{similar}{Similar Semigroups}
\newtheorem*{quotient}{Quotient Semigroups}
\newtheorem*{rescaled}{Rescaled Semigroups}

\newtheorem{thm}{Theorem}[section]
\newtheorem{prop}[thm]{Proposition}
\newtheorem{cor}[thm]{Corollary}
\newtheorem{lemma}[thm]{Lemma}
\newtheorem{defn}[thm]{Definition}
\numberwithin{equation}{section}

\theoremstyle{definition}
\newtheorem{rem}[thm]{Remark}
\newtheorem{ex}[thm]{Example}

\def\NN{{\mathbb N}}
\def\ZZ{{\mathbb Z}}
\def\RR{{\mathbb R}}
\def\QQ{{\mathbb Q}}
\def\C{{\textnormal {C}}}
\def\UC{{\textnormal {UC}}}
\def\Lip{{\textnormal {Lip}}}
\def\Id{{\textnormal {Id}}}

\DeclareSymbolFont{bbold}{U}{bbold}{m}{n}
\DeclareSymbolFontAlphabet{\mathbbold}{bbold}

\newcommand{\zs}

\newcommand{\goestru}{\xrightarrow{\tau_{ru}}}
\newcommand{\goesru}{\xrightarrow{ru}}

\DeclareMathOperator{\supp}{supp}

\makeindex

\begin{document}

\title[Relatively uniformly continuous semigroups]{Relatively uniformly continuous semigroups on vector lattices}
\author{M.~Kandi\' c}
\address{Faculty of Mathematics and Physics,
University of Ljubljana,
Jadranska 19,
SI-1000 Ljubljana,
Slovenija}
\address{Institute of Mathematics, Physics and Mechanics,
Jadranska 19,
SI-1000 Ljubljana,
Slovenija}
\email{marko.kandic@fmf.uni-lj.si}
\author{M.~Kaplin}
\address{Institute of Mathematics, Physics and Mechanics,
Jadranska 19,
SI-1000 Ljubljana,
Slovenija}
\address{Faculty of Mathematics and Physics,
University of Ljubljana,
Jadranska 19,
SI-1000 Ljubljana,
Slovenija}

\email{michael.kaplin@fmf.uni-lj.si}


\thanks{The first author acknowledges financial support from the Slovenian Research Agency, Grants No. P1-0222 and J1-8133}

\date{\today}

\begin{abstract}
In this paper we study continuous semigroups of positive operators on general vector lattices equipped with the relative uniform topology $\tau_{ru}$. We introduce the notions of strong continuity with respect to $\tau_{ru}$ and relative uniform continuity for semigroups. These notions allow us to study semigroups on non-locally convex spaces such as $L^p(\RR)$ for $0<p<1$ and non-complete spaces such as $\Lip(\RR)$, $\UC(\RR)$, and $\C_c(\RR)$. We show that the (left) translation semigroup on the real line, the heat semigroup and some Koopman semigroups are relatively uniformly continuous on a variety of spaces.
\end{abstract}

\maketitle
\section{Introduction and preliminaries} 
In the 1940s, E. Hille \cite{Hille:42,Hille:48} and K. Yosida \cite{Yosida:48} introduced the theory of strongly continuous semigroups on Banach spaces in order to treat evolution equations. By now, their theory is well established, and its applications reach well beyond the classical
field of partial differential equations. However, from the very beginning many situations occurred in which the 
underlying space is not a Banach space. In order to deal with such phenomena, already 
I. Miyadera \cite{Miyadera:59}, H. Komatsu \cite{Komatsu:64}, K. Yosida \cite{Yosida:65}, K.
Singbal-Vedak \cite{Singbal:65}, T. Komura \cite{Komura:68}, S. Ouchi \cite{Ouchi:73} and others generalized the
theory to strongly continuous semigroups on locally convex spaces.
Also, strongly continuous semigroups of positive operators on Banach lattices have
been discussed in \cite{Batty:84}, \cite{Arendt:86}, \cite{Batkai:17} and by others.

Our purpose is to provide a general framework for the theory of strongly continuous semigroups on vector lattices. Although vector lattices themselves are initially order and algebraic theoretical construct, they admit topologies which arise purely from order. The natural question that appears is whether one can study dynamical systems on general vector lattices. Since we want that our notion of strong continuity of semigroups on general vector lattices agrees with strong continuity for semigroups on Banach lattices, relative uniform topology $\tau_{ru}$ seems to be the correct choice.  This allows us to consider semigroups on non-Banach spaces, such as $\C_c(\RR), \Lip(\RR)$, $\UC(\RR)$, and $\C(\RR)$ or even on non-locally convex spaces such as $L^p(\RR)$ for $0<p<1$. We discuss two types of continuity of semigroups on vector lattices:  the strong continuity with respect to topology $\tau_{ru}$ and with respect to relative uniform convergence. The former notion is defined by $\tau_{ru}$-convergence and the latter is defined by relative uniform convergence.

%
%


This paper is structured as follows.
In \Cref{section RUC} we consider some general properties of $\tau_{ru}$ and we provide examples of  vector lattices together with corresponding relative uniform topologies.
In \Cref{section:tru semigroups} we introduce the notions of $\tau_{ru}$-strongly continuous semigroups and relatively uniformly continuous semigroups. We will see that the heat semigroup is relatively uniformly continous on $\Lip(\RR^N)$ and $\UC(\RR^N)$ for each $N \in \NN$. While relative uniform continuity implies $\tau_{ru}$-strong continuity, in general, these notions do not coincide as it is shown by the (left) translation semigroup on $L^p(\RR)$ $(0<p<\infty)$. In the rest of \Cref{section:tru semigroups} we present how one can lift strong continuity with respect to $ \tau_{ru}$ from a $ \tau_{ru}$-dense set to the whole space. Parallel theory for relatively uniformly continuous semigroups is considered in \Cref{ru BS}. A comparison between vector lattice case and Banach space case reveals that general vector lattices lack some property related to the principle of uniform boundedness. We introduce such property and call it ``property $(D)$". We prove that many important vector lattices posses it. This property enables us to provide the extension theorem for relatively uniformly continuous semigroups. This theorem is also fundamental for the proof of \cite[Theorem 5.4]{Kaplin:18} which is a Hille-Yosida type result for relatively uniformly continuous semigroups. In \Cref{Koopman} this extension theorem is used to identify relatively uniformly continuous Koopman semigroups on $\C(\RR)$ through their semiflows. Such semigroups have been studied by \cite{Koopman:31}, \cite{Koopman:42}, \cite{Ambrose:42}, \cite{Budisic:12},  \cite{terElst:15},\cite{Mauroy:16}, \cite{Batkai:17}, \cite{Edeko:18} and by others.
In \Cref{constructions} we study standard constructions of new relatively uniformly continuous semigroups from a given one.
%




Let us now recall some preliminary facts and notations that are needed throughout the text.
A family $(T(t))_{t\geq 0}$ of linear operators on a vector space $Y$ is a \emph{semigroup} if it satisfies the functional equation $$T(s+t)=T(t)T(s) \quad \text{for all } t,s\geq 0 \qquad \text{and} \qquad T(0)=I_Y.$$ If $Y$ is a Banach space, then we call a semigroup $(T(t))_{t\geq 0}$ a \emph{$C_0$-semigroup} or \emph{strongly continuous} on $Y$ when the operator $T(t)$ is bounded for each $t \geq 0$ and for each $y \in Y$ the \emph{orbit map} $$\zeta_y\colon t \mapsto \zeta_y(t)=T(t)y$$ is continuous with respect to the Euclidean topology $\tau_e$ on $\RR_+$ and the norm topology on $Y$.
If $Y$ is a vector lattice and $T(t)$ is a positive operator on $Y$ for each $t \geq 0$, then the semigroup $(T(t))_{t\geq 0}$ is called a \emph{positive semigroup}.

If $\tau$ is a linear topology on $Y$, then a net of linear operators $(T_\alpha)_{\alpha}$ is \emph{$\tau$-equicontinuous} when for each $\tau$-neighborhood of zero $V \subset Y$ there exists another $\tau$-neighborhood of zero $U\subset Y$ such that $T_\alpha U \subset V$ for all $\alpha$.  If for each $s>0$ the family of operators  $\{T (t)  \ \colon \ 0 \leq t \leq s \}$ is $\tau$-equicontinuous, then $(T(t))_{t\geq 0}$ is \emph{locally $\tau$-equicontinuous}.
If for $x \in X$, $s>0$ and a topology $\tau$ on $X$ we have $T(s+h)x \xrightarrow{\tau} T(s)x$ as $h \rightarrow 0$  and $T(h)x \xrightarrow{\tau} x $ as $h \searrow 0$, then we write ``$T(s+h)x \xrightarrow{\tau} T(s)x$ as $h \rightarrow 0$ for $s \geq 0$".

A net $(x_\alpha)_{\alpha} \subset X$ is \emph{relatively uniformly convergent to $x \in X$} if there exists some $u\in X$ such that for each $\varepsilon>0$ there exists $\alpha_0 $ such that $$|x_\alpha-x|\leq \varepsilon \cdot u$$ holds for all $\alpha\geq \alpha_0$. We call such an element $u\in X$ a \emph{regulator of $(x_\alpha)_{\alpha} $} and we write $x_\alpha \goesru x$. It is well-known that limits of relatively uniformly convergent sequences in $X$ are unique if and only if $X$ is Archimedean.
Throughout this paper, $X$ stands for an Archimedean vector lattice unless specified otherwise.

For the unexplained terminology about vector lattices and semigroups we refer the reader to \cite{Luxemburg:71} and \cite{Engel:00}, respectively.

%
%

\section{Relative uniform topology}\label{section RUC}

A subset $S$ of $X$ is called \emph{relatively uniformly closed}
whenever $(x_n)_{n \in \NN} \subset S$ and $x_n \goesru x$ imply $x \in S$. By \cite[Section 3]{Luxemburg:67}, the relatively  uniformly closed sets are exactly the closed sets of a certain topology in $X$, the \emph{relative uniform topology} which we denote by $\tau_{ru}$. This topology has been first studied by W.A.J. Luxemburg and L.C. Moore in \cite{Luxemburg:67}; see also \cite{Moore:68}. When a net $(x_\alpha)_{\alpha} \subset X$ converges to $x$ in $\tau_{ru}$, we write $x_\alpha \xrightarrow{\tau_{ru}} x$. Since $X$ is Archimedean, the topological space  $(X, \tau_{ru}) $ satisfies the $T_1$-separation axiom.

The following proposition yields that if one starts by defining closed sets through nets, one ends up with the same topology.

\begin{prop}\label{nets and sequences}
A subset $S$ of $X$ is relatively uniformly closed if and only if for each net $(x_\alpha)_{\alpha} \subset S$ and $x \in X$ with $x_\alpha \goesru x$ we have $x \in S$.
\end{prop}

\begin{proof}
It suffices to prove the ``only if" statement. Fix a relatively uniformly closed set $S \subset X$, $x \in X$ and a net $(x_\alpha)_{\alpha} \subset S$ satisfying $x_\alpha \goesru x$ with respect to some regulator $u$. We show that $x \in S$.
For each $n\in\mathbb N$ pick any index $\alpha_n$ such that $|x_{\alpha_n}-x|\leq \frac{1}{n}\cdot u$. Then $x_{\alpha_n}\goesru x$, and since $S$ is relatively uniformly closed, we conclude $x\in S$.
\end{proof}

We proceed with various examples of important vector lattices together with their relative uniform topologies and convergences which will be needed throughout the paper.

\begin{ex}\label{topology examples}\hfill
\begin{enumerate}[(a)]
\item On a vector lattice $X$ with an order unit $u \in X$ the relative uniform topology $\tau_{ru}$ is generated by the norm $$\|x\|_{u}:= \inf \{ \lambda >0 \ \colon \ |x| \leq \lambda \cdot u \},$$
since $x_\alpha \goesru x$ if and only if $x_\alpha \xrightarrow{\|\cdot\|_{u}} x$. Such vector lattices are $\tau_{ru}$-complete if and only if they are uniformly complete.
\item It is well-known that in a completely metrizable locally solid vector lattice $(X,\tau)$ every convergent sequence has a subsequence which converges relatively uniformly to the same limit, see \cite[Lemma 2.30]{Aliprantis:07}. This immediately yields that a subset of $X$ is relatively uniformly closed if and only if it is $\tau$-closed, so that topologies $\tau_{ru}$ and $\tau$ agree. In particular, if $X$ is a Banach lattice, then $\tau_{ru}$ agrees with norm topology.
\item For $0<p<1$ the vector lattice $L^p(\RR)$ equipped with the topology $\tau$ induced by the metric $$\qquad d_p(f_1,f_2):=\int_{\RR} |f_1(x)-f_2(x)|^p \ dx$$
    is a completely metrizable locally solid vector lattice which is not locally convex.
\item The vector lattice $C(\RR)$ equipped with the topology of uniform convergence on compact sets is a completely metrizable locally convex solid vector lattice.
\end{enumerate}
\end{ex}

In the following proposition we characterize relative uniform convergence in $\C_c(\RR)$.
\begin{prop}\label{compact support ru}
A net $(f_\alpha)_{\alpha} \subset \C_c(\RR)$ converges relatively uniformly to $f \in \C_c(\RR)$ if and only if $f_\alpha \xrightarrow{\| \cdot \|_{\infty}} f$ and there exists a compact set $K \subset \RR$ and $\alpha_0$ such that ${f_{\alpha}}|_{K^c}=0$ for all $\alpha \geq \alpha_0$.
\end{prop} 
\begin{proof}
($\Rightarrow$) Fix $\varepsilon >0$. There exist $u \in \C_c(\RR)$, independent of $\varepsilon$, and $\alpha_0$ such that
$$|f_\alpha -f| \leq \varepsilon \cdot u$$
for all $\alpha \geq \alpha_0$. This immediately implies that $$\| f_\alpha -f \|_{\infty} \leq  \varepsilon \cdot \| u \|_{\infty}\quad \text{ and } \quad| f_\alpha | \leq |f_\alpha -f| + |f| \leq \varepsilon \cdot u + |f|$$
for all $\alpha \geq \alpha_0$ and hence, $f_\alpha \xrightarrow{\| \cdot \|_{\infty}} f$ and ${f_{\alpha}}|_{K^c}=0$ for all $\alpha \geq \alpha_0$ where $K$ is the compact support of the function $\varepsilon \cdot u+|f|$. 

($\Leftarrow$) To construct a regulator $u \in \C_c(\RR)$, pick compact sets $K_1, K_2 \subset \RR$ and $\alpha_0$ such that $f_{|{K_1}^c}=0$ and ${f_{\alpha}}_{|{K_2}^c}=0$ hold for all $\alpha \geq \alpha_0$ and set $K:=K_1 \cup K_2$. Then for all $\alpha \geq \alpha_0$ we obtain 
$$(|f_{\alpha}-f|)_{|K^c} \leq (|f_{\alpha}|+|f|)_{|K^c}=0.$$ By assumption, for each $\varepsilon >0$ there exists $\alpha_1$ such that $\|f_\alpha -f\|_\infty \leq \varepsilon $
holds for all $\alpha \geq \alpha_1$. Hence, for any $\alpha \geq \alpha_0, \alpha_1$ we have
$$
|f_\alpha (x)-f(x)|\leq\left\{\begin{array}{lcl}
\varepsilon,& x \in  K,\\
0  ,& x \in K^c.
\end{array}\right.
$$
Now it is easy to see that any positive function $u \in \C_c(\RR)$  with $u(x)=1$ for all $x \in K$ regulates the convergence $f_\alpha \goesru f$. 
\end{proof} 

By \Cref{compact support ru}, a set $S \subset \C_c(\RR)$ is relatively uniformly closed if and only if for $(f_n)_{n \in \NN} \subset S$ and $f \in \C_c(\RR)$ the existence of a compact set $K \subset \RR$ such that ${f_{n}}|_{K^c}=0$ for all $n \in \NN$ and $f_n \xrightarrow{\| \cdot \|_{\infty}} f$ imply $f \in S$.

If a vector lattice $X$ has an order unit  $u$, \Cref{topology examples}(a) yields that $\tau_{ru}$ on $X$ agrees with the norm topology induced by the norm $\|\cdot\|_u$.
The following proposition shows that for each $N \in \NN$ the vector lattices of Lipschitz continuous functions $\Lip(\RR^N)$ and uniformly continuous functions $\UC(\RR^N)$ posses order units.

\begin{lemma}\label{1}
For each $f \in \UC(\RR^N)$ and each $\varepsilon>0$ there exists $\delta > 0$ such that \begin{equation}\label{delta}
|f(x)-f(y)| \leq \varepsilon \cdot  (\|x-y\| \cdot \delta^{-1} +1)
\end{equation}
holds for all $x,y \in \RR^N$. In particular, the function $u  \colon  x \mapsto 1+\|x\|$ is an order unit of vector lattices $\Lip(\RR^N)$ and $\UC(\RR^N)$.
\end{lemma}
\begin{proof}
To prove this, fix $f \in \UC(\RR^N), \varepsilon>0$ and find $\delta >0$ such that $|f(t)-f(s)| \leq \varepsilon$ whenever $\|t-s\| \leq \delta$. Pick $x,y \in \RR^N$ and set $z=\frac{x-y}{\|x-y\|}$. If $\|x-y\| \leq \delta$, then (\ref{delta}) holds. Assume now that $\|x-y\| > \delta$ and write $\|x-y\|= M \delta +r$ for some $M \in \NN$ and $0 \leq r < \delta$. Then
\begin{align*}\label{ucf}
|f(x)-f(y)| &\leq \sum_{n=1}^{M}|f\left( (n\delta+r) \cdot z +y\right)-f \left(((n-1)\delta+r ) \cdot z +y \right)| + |f(r \cdot z+y)-f(y)| \\
& \leq\varepsilon \cdot ( M+1) \leq \varepsilon \cdot (\|x-y \| \cdot \delta^{-1} +1).
\end{align*}
In particular, when $y =0$ and $\varepsilon =1$ we obtain
\begin{equation*}
|f(x)| \leq |f(x)-f(0)| +|f(0)| \leq  1+\|x\|\cdot \delta^{-1}  +|f(0)|\leq  ( 1+\delta^{-1}+|f(0)|) \cdot (1+ \|x\|),
\end{equation*}
 so that $u$ is a order unit for $\Lip(\RR^N)$  and $\UC(\RR^N)$.
\end{proof}

It is well-known that $x_n\goesru x$ implies $x_n\xrightarrow{\tau_{ru}}x$, see \cite[Section 3]{Luxemburg:67}.
%
%
While in general, the backward implication is not true, for sequences $\tau_{ru}$ convergence is equivalent to the following.
A sequence $(x_n)_{n \in \NN} \subset X$ is \emph{relatively uniformly $*$-convergent} to $x$ if every subsequence of $(x_n)_{n \in \NN} $ contains a further subsequence that is relatively uniformly convergent to $x$.
Similarly, a net $(x_\alpha)_{\alpha} \subset X$ is \emph{relatively uniformly $*$-convergent to $x \in X$} if every subnet of $(x_\alpha)_{\alpha} $ contains a further subnet that is relatively uniformly convergent to $x$. We write $x_\alpha \xrightarrow{ru^*} x$ if a net or a sequence $(x_\alpha)$ relatively uniformly $*$-converges to $x$.

It is clear that relative uniform convergence implies relative uniform $*$-convergence in case of sequences and nets.  Moreover, relative uniform $*$-convergence is tightly connected to $\tau_{ru}$-convergence. By \cite[Theorem 3.5]{Luxemburg:67}, $x_n \xrightarrow{ru^*} x$ is equivalent to $x_n \goestru x$. The natural question that appears here is what happens when one replaces sequences by nets. The following proposition shows that relative uniform $*$-convergence always implies $\tau_{ru}$-convergence.  On the other hand, \Cref{first uncountable ordinal} will show that the converse implication, in general, is not true.

\begin{prop}\label{sfjihfusdhif}
If $x_\alpha \xrightarrow{ru^*} x$, then $ x_\alpha \goestru x.$
\end{prop}

\begin{proof}
We first consider the special case when $x_\alpha\xrightarrow{ru}x.$
Fix an open $\tau_{ru}$-neighborhood $U \subset X$ for $x$ and $(x_\alpha)_{\alpha} \subset X$ with $x_\alpha \goesru x$ with respect to a regulator $u \in X$.

We claim that there exists $n \in \NN$ such that $|x_\alpha -x| \leq \frac{1}{n}\cdot u$ implies $x_\alpha \in U$. Assume otherwise. Then for each $n \in \NN$ there exists $\alpha_n$ such that $x_{\alpha_n} \not \in U$ and $|x_{\alpha_n} -x| \leq \frac{1}{n}\cdot u$. From $x_{\alpha_n} \goesru x$ we conclude $x_{\alpha_n} \goestru x$ which is a contradiction to $x_{\alpha_n} \not \in U$ for all $n \in \NN$. Hence, there exists $n \in \NN$ such that $|x_\alpha -x| \leq \frac{1}{n}\cdot u$ implies $x_\alpha \in U$. Since $x_\alpha \goesru x$, there exists $\alpha_0$ such that $|x_\alpha-x| \leq \frac{1}{n} \cdot u$ holds for all $\alpha \geq \alpha_0$ and hence, we have $x_{\alpha} \in U$ for all $\alpha \geq \alpha_0$

For the general case, assume that $x_\alpha \xrightarrow{ru^*} x$ while $x_\alpha \not \goestru x$. Then there exists an open $\tau_{ru}$-neighborhood $V \subset X$ of $x$ such that for each $\alpha$ there exists $\beta_\alpha \geq \alpha$ with $x_{\beta_\alpha} \not \in V$. We claim that $(x_{\beta_\alpha})_{\alpha}$ is a subnet of $(x_\alpha)_{\alpha}$. For each $\alpha_1$ and $ \alpha_2$ find $\alpha$ such that $\alpha \geq \beta_{\alpha_1}, \beta_{\alpha_2}$ and take $\beta_{\alpha} \geq \alpha$. Hence $(x_{\beta_\alpha})_{\alpha}$ is a net and by construction of $(\beta_\alpha)_\alpha$ it is a subnet of $(x_\alpha)_{\alpha}$. By assumption, there exists a subnet of $(x_{\beta_\alpha})_{\alpha}$ which converges relatively uniformly to $x$. This subnet necessarily $\tau_{ru}$-converges to $x$. This is a contradiction to $x_{\beta_\alpha} \not \in V$ for all $\alpha$.
\end{proof}

\begin{ex}\label{first uncountable ordinal}
Consider the first uncountable ordinal $\omega_1$. It is well-known that $\omega_1$ is an uncountable well-ordered set and all countable subsets of $\omega_1$ have suprema. This immediately yields that no cofinal subset of $\omega_1$ is countable.

Let $X$ be the vector lattice of all real functions on $\omega_1$ with countable support. By \cite[Example 2.2]{Moore:68}, the relative uniform topology on $X$ is the topology of pointwise convergence. Consider the net $(\chi_{\alpha})_{\alpha \in \omega_1}$ in $X$ where $\chi_{\alpha}$ is the characteristic function of $\{\alpha\}$. It is clear that $(\chi_{\alpha})_{\alpha \in \omega_1}$ converges pointwise to $0$. 

Assume that there exists a subnet $(\chi_{\beta})$ of $(\chi_\alpha)$ such that $\chi_{\beta} \goesru 0$. Then there exists $u \in X$ and $\beta_0$ such that $|\chi_{\beta}| \leq u$ for all $\beta \geq \beta_0$. Hence, for all $\beta \geq \beta_0$ we have $u(\beta)\not =0$. Since $\omega_1$ has no countable cofinal subsets, the set $\{\beta:\; \beta \geq \beta_0\}$ is uncountable, so that the support of $u$ is uncountable. This is absurd. \end{ex}

\section{Semigroups on $(X, \tau_{ru})$} \label{section:tru semigroups}

In this section we introduce two notions of continuity for semigroups on general vector lattices and provide the first examples such as the (left) translation and the heat semigroup. In \Cref{Koopman} we will encounter more examples. We show that these notions expand semigroup theory to spaces which are not locally convex or complete. In \Cref{L_p 0<p<infty} we will see that these notions truly differ. Furthermore, we provide conditions under which it is enough to check $\tau_{ru}$-continuity of a semigroup on a $\tau_{ru}$-dense set to obtain $\tau_{ru}$-continuity on the whole vector lattice. Finally, we will prove that on bounded time intervals each orbit map of a relatively uniformly continuous semigroup is order bounded.

A semigroup $(T(t))_{t\geq 0}$ on $X$ is \emph{strongly continuous with respect to $ \tau_{ru}$} or \emph{$\tau_{ru}$-strongly continuous} if for each $x\in X$ the orbit map $\zeta_x\colon (\mathbb R_+, \tau_e) \rightarrow (X, \tau_{ru})$ is continuous, i.e. $$\zeta_x(t+s) \underset{}{\goestru } \zeta_x(s)$$ for each $s \geq 0$ as $t \rightarrow 0$. If, in addition, we have $$\zeta_x(t+s) \underset{}{\goesru } \zeta_x(s)$$
for each $x \in X$ and $s \geq 0$ as $t \rightarrow 0$, then $(T(t))_{t\geq 0}$ is \emph{relatively uniformly continuous}. Since relative uniform convergence implies $\tau_{ru}$-convergence, every relatively uniformly continuous semigroup is $\tau_{ru}$-strongly continuous.
In the special case when $X$ is a Banach lattice, a positive semigroup $(T(t))_{t\geq 0}$ is $\tau_{ru}$-strongly continuous on $X$ if and only if it is a positive $C_0$-semigroup on $X$. We proceed by examples of relatively uniformly continuous semigroups.

For a function $f \colon  \RR \rightarrow \RR$ and $t \geq 0$, we consider the (left) translation operator
$$(T_l(t)f)(x)=f(t+x), \qquad x \in \RR$$
of $f$ by $t$. It is evident that by fixing a translation invariant space $Y$ of functions on $\RR$ one obtains a semigroup $(T_l(t))_{t \geq 0}$ on $Y$ which we call the \emph{(left) translation semigroup} on $Y$.

\begin{prop}\label{translation semigroup on Cc}
The (left) translation semigroup $(T_l(t))_{t\geq 0}$ is relatively uniformly continuous on $\C_c(\RR)$.
\end{prop}
\begin{proof}
 Fix $f \in \C_c(\RR)$, $\varepsilon>0$ and $s\geq 0$. Since $f$ is uniformly continuous on $\RR$, there exists $\delta>0$ such that
$\|T_l(s+h)f-T_l(s)f\|_\infty<\varepsilon$ for all $|h|<\delta.$ This proves $T_l(s+h)f\xrightarrow{\|\cdot\|_\infty}T_l(s)f$ in $\C_c(\RR)$ as $h\to 0.$

Since $T_l(s)f\in \C_c(\RR)$, there exists $n\in\mathbb N$ such that $\supp T_l(s)f\subseteq [-n,n].$ Choose any $h_0>0$. If $|h|<h_0$, then a direct computation shows that
$\supp T_l(s+h)f\subseteq [-n-h_0,n+h_0]$. An application of \Cref{compact support ru} concludes the proof.
\end{proof}
In the end of \Cref{Koopman} we will see that the (left) translation semigroup is also relatively uniformly continuous on the vector lattices $\Lip(\RR)$, $\UC(\RR)$ and $\C(\RR)$.
The following proposition, whose proof is postponed until we prove \Cref{relatively uniformly continuous semigroup 0}, shows that the heat semigroup is relatively uniformly continuous on $\Lip(\RR^N)$ and $\UC(\RR^N)$ for each $N \in \NN$ .
\begin{prop}\label{heat semigroup}
For a fixed $N \in \NN$ consider $X =\Lip(\RR^N)$ or $X=\UC(\RR^N)$ and the family of operators $(T(t))_{t\geq 0}$ defined by
\begin{align*}
(T(t)f)(y)=\frac{1}{(4 \pi t )^{N/2}} \int_ {\RR^N}^{} e^{-\frac{\|x\|^{2}}{4t}} \cdot  f(x+y) \, dx, \quad T(0)=I_{X},
\end{align*}
for each $t \geq 0, x \in \RR^N$ and $f \in X$. Then $(T(t))_{t\geq 0}$ is a relatively uniformly continuous positive semigroup on $X$.
\end{prop}

Our next goal is to weaken the hypothesis on the semigroup which assure that it is still $\tau_{ru}$-strongly continuous  or relatively uniformly continuous. The motivation comes from the general theory of $C_0$-semigroups on Banach spaces. As it is shown in \cite[Proposition I.5.3]{Engel:00}, a semigroup on a Banach space $Y$ is a $C_0$-semigroup if and only if it is norm bounded on bounded time intervals and the orbit maps $\zeta_y$ are continuous on a norm dense set of elements of $Y$. This result heavily relies on the principle of uniform boundedness which is unavailable in general vector lattices. The following theorem is a vector lattice version of the above result for strong continuity with respect to $\tau_{ru}$ in the case when $\tau_{ru}$ is a linear topology. By \cite[Theorem 2.1]{Moore:68}, the topology $\tau_{ru}$ is linear whenever it is first countable.

\begin{thm}\label{str cont on dense subset}
If $\tau_{ru}$ is a linear topology on $X$, then a semigroup $(T(t))_{t\geq 0}$ on $X$ is $\tau_{ru}$-strongly continuous  if and only if for each $s \geq 0$ the following two assertions hold.
\begin{itemize}
\item [\textup{(i)}] There exists a $\tau_{ru}$-dense subset $D$ of $X$ such that $T(s+t)y \goestru T(s)y$ as $t \rightarrow 0$ for each $y \in D$.
\item [\textup{(ii)}] For each net $(x_\alpha)_{\alpha} \subset X$ with $x_\alpha \goestru 0$ and each open $\tau_{ru}$-neighborhood of zero $V_0 \subset X$ there exists $\alpha_0$ and $\delta >0$ such that
$$T(s+t)x_{\alpha_0}-T(s)x_{\alpha_0} \in V_0$$
holds for all $t \in [-\delta,\delta]$ when $s>0$ and all $t \in [0,\delta]$ when $s=0$.
\end{itemize}
\end{thm}
\begin{proof}
Since the forward implication is clear we only prove the backward implication.

 Fix an open $\tau_{ru}$-neighborhood of zero $V_0 \subset X$ and take any open $\tau_{ru}$-neighborhood of zero $V_1 \subset X$ such that $V_1+V_1 \subset V_0$. Fix $s \geq 0$ and $ x \in X$. By (i), there exists a net $(x_\alpha)_{} \subset X$ such that $x_\alpha \goestru x$ and $T(s+t)x_\alpha \goestru T(s)x_\alpha$ as $t \rightarrow 0$ for each $\alpha$. Hence, by (ii), there exist $\alpha_0$ and $\delta >0$ such that
$$T(s+t)(x-x_{\alpha_0})-T(s)(x-x_{\alpha_0}) \in V_1 \
\text{  and  } \ T(s+t)x_{\alpha_0}-T(s)x_{\alpha_0}\in V_1$$
hold for all $t \in [- \delta, \delta]$ when $s>0$ and all $t \in [0, \delta]$ when $s=0$. Therefore,
\begin{align*}
T(s+t)x-T(s)x &= [T(s+t)(x-x_{\alpha_0})-T(s)(x-x_{\alpha_0})]\\
&+[T(s+t)x_{\alpha_0}-T(s)x_{\alpha_0}] \in V_1+V_1 \subset V_0
\end{align*}
holds for all $t \in [- \delta, \delta]$ when $s>0$ and all $t \in [0, \delta]$ when $s=0$. This proves that $(T(t))_{t\geq 0}$ on $X$ is $\tau_{ru}$-strongly continuous.
\end{proof}

In \Cref{ru BS} we will establish an analogous version of \Cref{str cont on dense subset} for relatively uniformly continuous semigroups on a particular class of vector lattices which allow a version of the principle of uniform boundedness.

The importance of the following corollary lies in its applicability. For locally $\tau_{ru}$-equicontinuous semigroups, $\tau_{ru}$-strong continuity is equivalent to $\tau_{ru}$-strong continuity at zero.

\begin{cor}\label{equicont to strongly continuous}
Let $\tau_{ru}$ be a linear topology on $X$ and $(T(t))_{t\geq 0}$ a locally $\tau_{ru}$-equicontinuous semigroup on $X$. Then $(T(t))_{t\geq 0}$ is $\tau_{ru}$-strongly continuous  if and only if there exists a $\tau_{ru}$-dense subset $D$ of $X$ such that $T(h)y \goestru y$ as $h \searrow 0$ for each $y \in D$.
\end{cor}

\begin{proof}
It suffices to prove the ``only if" statement.
We will check (i) and (ii) from \Cref{str cont on dense subset}. Fix $s>0$ and an open $\tau_{ru}$-neighborhood of zero $V_0 \subset X$. There exists an open $\tau_{ru}$-neighborhood of zero $V_1 \subset X$ such that $V_1+V_1 \subset V_0$. Since $(T(t))_{t\geq 0}$ is a locally $\tau_{ru}$-equicontinuous semigroup on $X$, there exists a symmetric open $\tau_{ru}$-neighborhood of zero $U \subset X$ such that $T(t)U \subset V_1$ for all $t \in [0,s+1]$.

(i) Assume that there exists a $\tau_{ru}$-dense subset $D$ of $X$ such that for each $y \in D$ we have $T(h)y \goestru y$ as $h \searrow 0$. Fix $y \in D$. There exists $ \delta >0 $ such that $T(h)y-y \in U$ for all $h \in [0 , \delta]$. Then
\begin{align*}
T(s+h)y-T(s)y =T(s)(T(h)y-y) \in V_1 \subset V_0
\end{align*}
holds for all for all $h \in [0 , \delta]$ and
\begin{align*}
T(s-h)y-T(s)y =T(s-h)(y-T(h)y) \in V_1 \subset V_0
\end{align*}
holds for all $h \in [0 , \min\{\delta,s\}]$ when $s >0$. This proves (i).

(ii) Pick a net $(x_\alpha)_{\alpha} \subset X$ with $x_\alpha \goestru 0$ and find $\alpha_0$ such that $x_{\alpha_0} \in U$. If $0<h<1$, then $T(s+h)x_{\alpha_0}\in V_1$. If $0<h<s$, then again $T(s-h)x_{\alpha_0}\in V_1$. Hence, if $h$ satisfies $|h|<\min\{1,s\}$, we have
$$T(s+h)x_{\alpha_0}-T(s)x_{\alpha_0} \in V_1+ V_1 \subset V_0$$ which completes the proof. 
\end{proof}
The following proposition shows that the notion of relative uniform continuity is, in general, stronger than the notion of strong continuity with respect to $\tau_{ru}$. Furthermore, it also provides an example of a completely metrizable locally solid vector lattice $(X,\tau_{ru})$ that is not locally convex and a $\tau_{ru}$-strongly continuous semigroup on $X$ which is not relatively uniformly continuous.

Having in mind that $\tau_{ru}$ on Banach lattices agrees with norm topology, in the case $1\leq p<\infty$ the following proposition recovers \cite[Example I.5.4]{Engel:00}.
\begin{prop}\label{L_p 0<p<infty}
For each $0<p<\infty$ the (left) translation semigroup $(T_l(t))_{t\geq 0}$ on $L^p(\RR)$ is $\tau_{ru}$-strongly continuous but not relatively uniformly continuous.
\end{prop}

\begin{proof}
Pick $0<p<\infty$ and denote by $\tau$ the topology induced by the metric $d_p$ on $L^p(\RR)$.
Since $(L^p(\RR), \tau)$ is a completely metrizable locally solid vector lattice, $\tau$ and $\tau_{ru}$ agree on $L^p(\RR)$ by \Cref{topology examples}(c). The same arguments as in the classical sense show that $\C_c(\RR)$ is dense in $(L^p(\RR), \tau_{ru})$.

Since $(T(t))_{t\geq 0}$ is relatively uniformly continuous on $\C_c(\RR)$ by \Cref{translation semigroup on Cc}, it is also $\tau_{ru}$-strongly continuous. Furthermore, for each $t \geq 0$ the operator $T_l(t)$ preserves every open $d_p$-ball with center at zero, from where it follows that the semigroup $(T_l(t))_{t\geq 0}$ is locally $\tau_{ru}$-equicontinuous on $L^p(\RR)$. By \Cref{equicont to strongly continuous}, we conclude that $(T_l(t))_{t\geq 0}$ is $\tau_{ru}$-strongly continuous on $L^p(\RR)$.
%

To show that $(T_l(t))_{t\geq 0}$ is not relatively uniformly continuous on $L^p(\RR)$, consider the function $$f\colon x \mapsto \left|\frac{1}{x-\frac{1}{2}}\right|^{\frac{1}{2p}}$$ in $L^p(\RR)$. Assume that there exist a function $u \in L^p(\RR)$ and $0<\delta <\frac{1}{2}$ such that $|T_l(t)f-f| \leq u$ holds in $L^p(\RR)$ for all $t \in [0, \delta]$, i.e., $$\left|\left|\frac{1}{t+x-\frac{1}{2}}\right|^{\frac{1}{2p}}- \left|\frac{1}{x-\frac{1}{2}}\right|^{\frac{1}{2p}}\right| \leq u(x)$$ holds for all $t \in [0, \delta]$ and almost every $x \in \RR$. Hence, the family $\mathcal F=\{T_l(t)f:\; 0\leq t\leq \delta\}$ is bounded above
and since $L^p(\RR)$ is Dedekind complete, there exists
$g:=\sup\{T_l(t)f:\; 0\leq t\leq \delta\}$ in $L^p(\RR)$. This is impossible since $g$ attains infinity on a set of positive measure.
%
%
\end{proof}
The remaining part of this paper is devoted to relatively uniformly continuous semigroups of positive operators. In \Cref{ru BS} we will prove \Cref{ruc on dense to ruc on whole space} which is a version of \Cref{equicont to strongly continuous} for relatively uniformly continuous semigroups.
In the case of $\tau_{ru}$-strongly continuous semigroups, we were able to provide such result only for locally $\tau_{ru}$-equicontinuous semigroups on $X$. The reason behind is not so surprising. Consider a semigroup of bounded operators on a Banach lattice. Since $\tau_{ru}$ agrees with norm topology, local $\tau_{ru}$-equicontinuity agrees with local equicontinuity which is equivalent to uniform boundedness of the semigroup on bounded time intervals. Furthermore, applying the principle of uniform boundedness, the latter is equivalent to the fact that the semigroup is pointwise bounded on bounded time intervals. For more details see \cite[Proposition I.5.3]{Engel:00}. In the case of relatively uniformly continuous semigroups we conclude that relatively uniformly continuous positive semigroups are pointwise order bounded on bounded time intervals.

\begin{prop}\label{uniform bound}
Suppose $(T(t))_{t\geq 0}$ is a positive semigroup on $X$ such that for each $x\in X$ we have $T(h)x\goesru x$ as $h\searrow 0$. Then for each $s \geq 0$ and $x \in X$ the set $$\{|T(t)x| \ \colon \; 0\leq t\leq s\}$$ is order bounded in $X$.
\end{prop}

\begin{proof}
Fix $s \geq 0$ and $x \in X$. There exist $u\in X_+$ and $\delta>0$ such that for all $0\leq h\leq \delta$ we have $|T(h)x-x|\leq  u.$
Pick $t\in [0,s]$ and find $n\in \mathbb N_0$ and $0\leq h<\delta$ such that $t=n\delta+h.$ Then
$$|T(t)x|=|T(\delta)^n T(h)x|\leq T(\delta)^n(|x|+u).$$
Let $n_0$ be the smallest positive integer such that $n_0\delta \geq s.$ If we define
$$v:=\bigvee_{k=0}^{n_0} T(\delta)^k(|x|+u) \in X_+,$$
we have $\{|T(t)x|\ \colon\; 0\leq t\leq s\} \subset [0,v]$.
\end{proof}


The following result is a version of \cite[Proposition I.5.3]{Engel:00} for relatively uniformly continuous semigroups of positive operators. It says that a semigroup is relatively uniformly continuous  if and only if it is relatively uniformly continuous at $t=0$ and positive $x$.

\begin{prop}\label{relatively uniformly continuous semigroup 0}
Let $(T(t))_{t\geq 0}$ be a positive semigroup on $X$. Then $(T(t))_{t\geq 0}$ is relatively uniformly continuous on $X$ if and only if $T(h)x \underset{}{\goesru } x$ as $h \searrow 0$ for positive vectors $x \in X_+$.
\end{prop}
\begin{proof}
Only the ``if statement" requires a proof. Fix $s > 0$ and $x \in X$. Let $u$ be one of the regulators for $T(h)x_+\goesru x_+$ and $T(h)x_-\goesru x_-$ as $h\searrow 0$. Pick $\varepsilon>0$ and find $0<\delta < s$ such that for all $h \in [0,\delta]$ we have
$|T(h)x_+-x_+|\leq \frac{\varepsilon}{2} \cdot u$ and $|T(h)x_--x_-|\leq \frac{\varepsilon}{2} \cdot u.$ By \Cref{uniform bound} we can also find $v \in X_+$ such that $T(h)u\leq v$ for all $h \in [0,s]$. Then
\begin{align*}
|T(s+h)x-T(s)x|&\leq T(s) (|T(h)x_+-x_+| + |T(h)x_--x_-|)\\
& \leq \varepsilon \cdot T(s)u \leq \varepsilon \cdot v
\end{align*}
and, similarly,
\begin{align*}
|T(s-h)x-T(s)x|&\leq T(s-h)(|T(h)x_+-x_+| + |T(h)x_--x_-|)\\
&\leq \varepsilon \cdot T(s-h)u \leq \varepsilon \cdot v
\end{align*}
hold for all $h \in [0,\delta]$. This proves that $(T(t))_{t\geq 0}$ is relatively uniformly continuous on $X$.
\end{proof}
To see the usefulness of the preceding result we apply it in the following proof.

\begin{proof}[Proof of \Cref{heat semigroup}]
 We prove first that $\Lip(\RR^N)$ and $\UC(\RR^N)$ are invariant under $T(t)$ for any $t \geq 0$. Fix $f \in \UC(\RR^N)$ and estimate
\begin{equation}\label{uniformly}
| (T(t)f)(y)-(T(t)f)(z)|  \leq \frac{1}{(4 \pi t )^{N/2}} \int_ {\RR^N}^{} e^{-\frac{\|x\|^{2}}{4t}} \cdot  |f(x+y)-f(x+z)| \, dx 
\end{equation}
for all $y,z \in \RR^N$. Pick $0<\varepsilon <1$ and find $\delta >0$ such that $|f(y)-f(z)| \leq \varepsilon$ for all $y,z \in \RR^N$ with $\|y-z\| \leq \delta$. Then the right hand side of (\ref{uniformly}) is smaller than or equal to $\varepsilon$ for all $y,z \in \RR^N$ with $\|y-z\| \leq \delta$ which proves that $T(t)f \in \UC(\RR^N)$. If, in addition, $f \in \Lip(\RR^N)$ with Lipschitz constant $K_f$, then the right hand side of (\ref{uniformly}) is smaller than or equal to $K_f  \|y-z\|$ for all $y,z \in \RR^N$ which proves that $T(t)f \in \Lip(\RR^N)$.

The semigroup property and positivity of the family $(T(t))_{t\geq 0}$ follow from its formula. We now prove that $(T(t))_{t\geq 0}$ is relatively uniformly continuous. To see this, fix $\varepsilon >0$ and set $C_N:=2  \cdot \frac{\Gamma \left (\frac{N+1}{2} \right)}{\Gamma \left (\frac{N}{2} \right)}$ where $\Gamma$ denotes the gamma function. By \Cref{1}, we can find $\delta >0$ such that 
$$|f(x+y)-f(y)|\leq   \frac{\varepsilon}{2} \cdot ( \|x\|\cdot  { \delta}^{-1}+1)$$
holds for all $x,y \in \RR^N$ from where it follows that
\begin{align*}
|(T(t)f)(y)-f(y)|&  \leq \frac{1}{(4 \pi t )^{N/2}} \int_ {\RR^N}^{} e^{-\frac{\|x\|^{2}}{4t}} \cdot  |f(x+ y) - f(y)| \, dx  \\
 & \leq \frac{1}{(4 \pi t )^{N/2}} \int_ {\RR^N}^{} e^{-\frac{\|x\|^{2}}{4t}} \cdot  \frac{\varepsilon}{2}  \cdot ( \|x\| \cdot  { \delta}^{-1} +1) \, dx   = \frac{\varepsilon}{2}   \cdot \left ( \sqrt{t}  \cdot  C_N \cdot { \delta}^{-1}  +1  \right)
\end{align*}
holds for all $t \geq 0$ and $y \in \RR^N$. Hence, we obtain 
\begin{align*}
|T(h)f-f| \leq \varepsilon \cdot \mathbbold{1}
\end{align*}
for all $h \in \left [0,  {C_N}^{-2} \cdot \delta^{2} \right]$. An application of \Cref{relatively uniformly continuous semigroup 0} concludes the proof.
\end{proof}

\section{Standard constructions}\label{constructions}

In this section we construct different relatively uniformly continuous semigroups from a given one. All the constructions are motivated by \cite[Chapter I.5.b]{Engel:00}. To prove that a given semigroup is relatively uniformly continuous we will tacitly use \Cref{relatively uniformly continuous semigroup 0}. For the sake of clarity, in this section $(T(t))_{t \geq 0}$ always denotes a given relatively uniformly continuous positive semigroup on a vector lattice $X$.

\begin{similar}
Let $V\colon Y \rightarrow X$ be a lattice isomorphism between vector lattices $X$ and $Y$. Then $\mathcal{S}:=(V^{-1}T(t) V)_{t \geq 0}$ is a relatively uniformly continuous positive semigroup on $Y$.
\end{similar}

\begin{proof}
It is easy to see that $\mathcal{S}$ is a positive semigroup on $Y$.
To prove that $\mathcal{S}$ is relatively uniformly continuous on $Y$, pick  $y \in Y$ and $\varepsilon >0$. Due to relative uniform continuity of $(T(t))_{t\geq 0}$ there exist $u \in X_+$ and $\delta >0$ such that
$$|T(h)Vy -Vy|_X \leq \varepsilon\cdot  u$$
holds for all $h \in [0,\delta]$. Since $V^{-1}\colon X\to Y$ is a lattice homomorphism, we obtain
\begin{align*}
|V^{-1}T(h)Vy-y|_Y&=
V^{-1}|T(h)Vy -Vy|_X\leq \varepsilon \cdot (V^{-1}u)
\end{align*}
for all $h \in [0,\delta]$.
\end{proof}

Next we consider semigroups on quotient vector lattices.
Let $J$ be an ideal in $X$ and let $\pi\colon X\to X/J$ be the quotient projection between vector lattices $X$ and $X/J$. In order to guarantee that $X/J$ is Archimedean, we require our ideal $J$ to be relatively uniformly closed (see \cite[Theorem 5.1]{Luxemburg:67}). 

\begin{quotient}
Suppose  $J$ is a relatively uniformly closed ideal which is invariant under operator $T(t)$ for each $t \geq 0$. Then the family of operators $(\widetilde{T}(t))_{t \geq 0}$ defined by
\begin{equation*}
\widetilde{T}(t) \pi (x)= \pi (T(t)x)
\end{equation*}
for each $x \in X$ and $t \geq 0$ is a relatively uniformly continuous positive semigroup on $X / J$.
\end{quotient}

\begin{proof}
It is easy to check that $(\widetilde{T}(t))_{t \geq 0}$ is a positive semigroup on $X / J$. To prove that $(\widetilde{T}(t))_{t \geq 0}$  is relatively uniformly continuous on $X / J$, pick $x \in X$ and $u\in X_+$ which regulates $T(h)x\goesru x$ as $h\searrow 0$. Pick $\varepsilon >0$ and find $\delta >0$ such that for all $h \in [0,\delta]$ we have
$|T(h)x -x| \leq \varepsilon \cdot u.$
Since $\pi $ is a lattice homomorphism, we obtain
$$|\widetilde{T}(h)\pi (x) -\pi(x)| = |\pi (T(h)x) -\pi(x)| = \pi( |(T(h)x -x| ) \leq \varepsilon \cdot \pi (u)$$
for all $h \in [0,\delta]$.
\end{proof}

The next standard construction on our list are   \emph{rescaled semigroups}.

\begin{rescaled}
For any numbers $\mu\in\RR$ and $\alpha>0$, the rescaled semigroup $(S(t))_{t\geq 0}$ defined by
$$S(t):=e^{\mu t}T(\alpha t)$$ is relatively uniformly continuous.
\end{rescaled}

\begin{proof}
A direct computation shows that $(S(t))_{t\geq 0}$ is a positive semigroup. To prove that $(S(t))_{t \geq 0}$  is relatively uniformly continuous on $X$, pick $x \in X$ and find $u\in X_+$ which regulates $T(h)x\goesru x$ as $h\searrow 0$.  Given any  $\varepsilon >0$, there
exists $\delta_1>0$ such that for all $h \in [0,\delta_1]$ we have
$|T(h)x -x| \leq \varepsilon \cdot u.$ Since the function $h\mapsto e^{\mu h}$ is continuous, there exists $\delta_2>0$ such that for all $h\in [0,\delta_2]$ we have $|e^{\mu t}-1|<\varepsilon.$ For $\delta:=\min\{\frac{\delta_1}{\alpha},\delta_2,1\}$ we obtain
\begin{align*}
|S(h)x-x|&\leq e^{\mu h}\cdot |T(\alpha h)x-x|+(e^{\mu h}-1)|x|\leq \varepsilon \cdot (e^{|\mu|}u+|x|)
\end{align*}
for each $0\leq h\leq \delta.$
\end{proof}

Next we deal with product semigroups. It is worth pointing out that the proof in our case is more complicated than the proof in the case of $C_0$-semigroups on Banach spaces.

\begin{product}
Let $(T(t))_{t \geq 0}$ and $(S(t))_{t \geq 0}$ be relatively uniformly continuous positive semigroups such that
$$T(t)S(t) = S(t)T(t)$$
holds for all $t \geq 0$.
Then $(T(t)S(t))_{t \geq 0}$ is a relatively uniformly continuous positive semigroup.
\end{product}

\begin{proof}
We prove first that $(T(t)S(t))_{t \geq 0}$ is a semigroup. As in \cite[I.5.15]{Engel:00} one can show that $T(q_1)S(q_2) = S(q_2)T(q_1)$ holds for all $q_1, q_2 \in \QQ_+$. Fix $t,s>0$ and $x \in X$. Find $u$ which regulates $T(t')x\goesru T(t)x$, $T(s')x\goesru T(s)x$, $T(t')S(s)x\goesru T(t)S(s)x$ and $S(s')T(t)x\goesru S(s)T(t)x$ as $t'\to t$ and $s'\to s$.  Pick $\varepsilon> 0$ and find  $0<\delta <1$ such that for all $t' \in [ t, t+\delta]$ and $s' \in [ s, s+\delta]$ we have
\begin{align*}
|T(t')x-T(t)x| \leq \tfrac{\varepsilon}{2} \cdot u,& \qquad |S(s')x-S(s)x| \leq \tfrac{\varepsilon}{2} \cdot u,\\
|T(t')S(s)x-T(t)S(s)x| \leq \tfrac{\varepsilon}{2} \cdot u,& \qquad |S(s')T(t)x-S(s)T(t)x| \leq \tfrac{\varepsilon}{2} \cdot u.
\end{align*}
 By \Cref{uniform bound}, we can find $v\in X_+$ such that
$$T(t') u \leq v \quad \text{ and } \quad S(s')u \leq v$$ hold for all $ t' \in [t, t+1] $ and $s' \in [s, s+1]$. Pick $t' \in [ t, t+\delta] \cap \QQ$ and $s' \in [s, s+\delta] \cap \QQ$. Since $X$ is Archimedean and for each $\varepsilon>0$ we have
\begin{align*}
|T(t)S(s)x-S(s)T(t)x| &\leq |T(t)S(s)x-T(t')S(s)x| + |T(t')S(s)x-T(t')S(s')x| \\
& + |T(t')S(s')x-S(s')T(t')x|+|S(s')T(t')x-S(s')T(t)x|\\
&+ |S(s')T(t)x-S(s)T(t)x|\\
& \leq \tfrac{\varepsilon}{2}\cdot u + T(t') |S(s)x-S(s')x| +S(s')|T(t')x-T(t)x| + \tfrac{\varepsilon}{2} \cdot u\\
& \leq \varepsilon \cdot u + \tfrac{\varepsilon}{2} \cdot ( T(t') u+ S(s')u ) \leq \varepsilon \cdot (u+v),
\end{align*}
we conclude $T(t)S(s)x=S(s)T(t)x.$ Since this holds for each $x\in X$, we obtain
$T(t)S(s)=S(s)T(t)$.
Now it is easy to deduce that $(T(t)S(t))_{t \geq 0}$ is a positive semigroup.

In order to prove that $(T(t)S(t))_{t\geq 0}$ is relatively uniformly continuous on $X$, we first find $u\in X_+$ such that for each $\varepsilon>0$ there exists $0< \delta<1$ such that
$$|T(h)x-x|\leq \varepsilon \cdot u \qquad \text{ and } \qquad |S(h)x-x|\leq \varepsilon \cdot u $$ hold for all $h\in [0,\delta].$ By \Cref{uniform bound}, there exists $ v \in X_+$ such that
$ T(h)u \leq v$ holds for all $h \in [0,1]$. By combining all of the above we see that
$$|T(h)S(h)x-x|\leq T(h)|S(h)x-x| + |T(h)x-x| \leq \varepsilon \cdot (v + u)$$
holds for all $h \in [0, \delta]$.
\end{proof}

We finish this section with short comments on the \emph{subspace} and \emph{adjoint} semigroups.

By \cite[Chapter I.5.12]{Engel:00}, if $J$ is a subspace of a Banach space $Y$ and $(T(t))_{t \geq 0}$ is a $C_0$-semigroup on $Y$ such that $T(t)J \subset J$ for each $t \geq 0$, then the restrictions $\widetilde T(t):= T(t)|_{J}$ form a $C_0$-semigroup $(\widetilde T(t))_{t \geq 0}$ on $J$.
In general, such construction does not apply to relatively uniformly continuous semigroups as the following example shows.

\begin{ex}
Consider the (left) translation semigroup on $\C(\RR)$. It is obvious that every operator from the semigroup leaves the ideal $\C_b(\RR)$ in $\C(\RR)$ invariant. Since $\C_b(\RR)$ has an order unit, relative uniform convergence agrees with norm convergence, so that in this case the (left) translation semigroup is relatively uniformly continuous if and only if it is a $C_0$-semigroup. It is well-known, however, that the (left) translation semigroup on $\C_b(\RR)$ is not a $C_0$-semigroup. On the other hand, in the end of \Cref{Koopman} we will prove that the (left) translation semigroup is relatively uniformly continuous on $\C(\RR)$.
\end{ex}

When comparing to Banach space case, the distinction  from the previous example appears when one wants to consider $\tau_{ru}$ on a sublattice $Y$ of a given lattice $X$. In the normed case, norm topology on the subspace agrees with the relative topology induced by the norm of the space. The relative topology on $Y$ induced by $\tau_{ru}$ from $X$ can be different than the relative uniform topology that $Y$ can induce on itself.

The situation with the adjoint semigroup is even subtler. If $X$ is a general vector lattice, probably the most natural candidate for the dual space is the order dual $X^\sim$ which, in fact, can be trivial (see  \cite[Theorem 25.1]{Zaanen:97}). If $X$ is a Banach lattice, then $X^\sim=X^*$ and so one can consider the adjoint semigroup on $X^*$. In the Banach space case, the most common example of a positive $C_0$-semigroup whose adjoint semigroup is not a $C_0$-semigroup is the (left) translation semigroup on $L^1(\RR)$. However, neither the (left) translation semigroup on $L^1(\RR)$ nor its adjoint semigroup which is the (right) translation semigroup on $L^\infty(\RR)$ is relatively uniformly continuous.

\section{The property $(D)$} \label{ru BS}

In this section we present an analogous version of \Cref{str cont on dense subset} and  \Cref{equicont to strongly continuous} for relatively uniformly continuous semigroups of positive operators. If one translates \Cref{str cont on dense subset} to the setting of $C_0$-semigroups on Banach spaces, one instantly realizes that without the principle of uniform boundedness, in general, there is no hope of having such a result. Therefore, we need to restrict ourselves to a particular case of vector lattices. We will say that such vector lattices have the \emph{property $(D)$} (see \Cref{property $(D)$ def}). Before we formally introduce this property and elaborate more on its connection to the principle of uniform boundedness in the classical setting of Banach spaces, we need to recall some basic notions and facts on relative uniform convergence and topology.

A set $D \subset X$ is \emph{ru-dense} if for each $x \in X$ there exists a sequence $(x_n)_{n \in \NN} \subset D$ such that $x_n \goesru x$. A similar argument as in the proof of \Cref{nets and sequences} shows that a set $D \subseteq X$ is ru-dense in $X$ if and only if for each $x\in X$ there is a net $(x_\alpha)$ in $D$ such that $x_\alpha\goesru x.$ Since $x_\alpha\goesru x$ implies $x_\alpha\goestru x$, a set $D\subseteq X$ is $\tau_{ru}$-dense in $X$ whenever it is ru-dense in $X$. The converse implication holds for vector lattices whose $\tau_{ru}$-topologies are completely sequential. Recall that $\tau_{ru}$-topology is said to be  \emph{completely sequential} if for any $S \subset X$ and any vector $x$ in the $\tau_{ru}$-closure of $S$ there exists a sequence in $S$ which converges relatively uniformly to $x$. Since in case of sequences $\tau_{ru}$ agrees with relative uniform $*$-convergence, it should be clear that $\tau_{ru}$-density in the completely sequential case implies ru-density. Before we close this short discussion on ru-density we would like to mention that $\tau_{ru}$ is completely sequential whenever $\tau_{ru}$ is first countable (see \cite[Theorem 2.1]{Moore:68}).


Now we are finally able to introduce the appropriate class of vector lattices which enables us to develop the remaining part of \cite[Proposition I.5.3]{Engel:00} for relatively uniformly continuous semigroups.

\begin{defn}\label{property $(D)$ def}
A vector lattice $X$ has property $(D)$ if for each net of linear operators $(T_\alpha)_{\alpha}$ on $X$ the following two assertions imply $T_\alpha x \goesru 0$ for each $x \in X$.
\begin{enumerate}
\item [\textup{(a)}] There exists an $ru$-dense subset $D \subset X$ such that $T_\alpha y \xrightarrow{ru} 0$ for each $y \in D$.
\item [\textup{(b)}] For each sequence $(x_n)_{n \in \NN} \subset X $ with $x_{n} \xrightarrow{ru} 0$ there exists $ u \in X_+$ such that for each $ \varepsilon >0$ there exist $N_\varepsilon \in \NN$ and $\alpha_\varepsilon $ such that $$|T_\alpha x_{n}| \leq \varepsilon \cdot u$$
holds for all $n \geq N_\varepsilon$ and $\alpha \geq\alpha_\varepsilon $.
\end{enumerate}
\end{defn}

Although \Cref{property $(D)$ def} at the first glance seems odd, it properly and nicely substitutes the principle of uniform boundedness from Banach spaces. Indeed, it can be easily shown that a sequence $(T_n)_{n\in\mathbb N}$ of operators on a Banach space $X$ converges pointwise to zero on $X$ if and only if it is uniformly bounded and it converges pointwise to zero on a dense subset of $X$. This is not true for nets since there are easy examples of norm unbounded nets of operators which converges pointwise to the zero operator.
If, however, in \Cref{property $(D)$ def} we replace our net $(T_\alpha)_\alpha$ with a $C_0$-semigroup $(T(t))_{t\geq 0}$ on a Banach space, then \cite[Proposition I.5.3]{Engel:00} yields that the semigroup $(T(t))_{t\geq 0}$ is norm bounded on time bounded intervals which relates to (b) from \Cref{property $(D)$ def}.

The main result of this section is \Cref{ruc on dense to ruc on whole space} which is an analogous version of \Cref{equicont to strongly continuous} for relatively uniformly continuous semigroups. Before we prove it, we will discuss two different classes of vector lattices which satisfy property $(D)$. One of them is the class of  completely metrizable locally solid vector lattices and the other one is the class of  vector lattices which satisfies condition $(R)$. Following \cite[Definition VI.5.1]{Vulikh:67}, a vector lattice $X$ satisfies \emph{condition $(R)$} whenever for each sequence $(u_{n})_{n \in \NN} \subset X_+$ there exists a sequence of positive scalars $(\lambda_{n})_{n \in \NN}$  such that $(\lambda_{n}u_{n})_{n \in \NN}$  is order bounded.  Vulikh introduced condition $(R)$ in order to study  order convergence.  From the proof of
\cite[Theorem  VI.5.2]{Vulikh:67} Swartz \cite{Swartz:88}
extracted the following property. A vector lattice $X$ is said to have \emph{property $(C)$} whenever each countable set of relatively uniformly convergent sequences has a common regulator. Due to our best knowledge, it seems that it remained unnoticed in the literature that property $(C)$ and condition $(R)$ are equivalent.
%

\begin{prop}\label{property (C) iff regular}
A vector lattice $X$ has property $(C)$ if and only if $X$ satisfies condition $(R)$.
\end{prop}

\begin{proof}
($\Rightarrow$) Fix a sequence $(u_{n})_{n \in \NN} \subset X_+ $ and for each $n \in \NN$ define the relatively uniformly converging sequence $(x_{m}^{(n)})_{m \in \NN}$ by $x_m^{(n)}:=\frac{1}{m}u_n$. Since $X$ has property $(C)$, there exists a positive vector $u\in X_+$ such that for each $n\in\NN$ the vector $u$ regulates the convergence $x_{m}^{(n)}\goesru 0$ as $m\to\infty$. Hence, for each $n \in \NN$ there exists $M_n\in \NN$ with $$\frac{1}{M_n}u_n \leq u.$$ This proves that $X$ satisfies condition $(R)$.

($\Leftarrow$) Here we follow the proof of \cite[Theorem VI.5.2]{Vulikh:67}. Fix sequences $(x_{n})_{n \in \NN}, (u_{n})_{n \in \NN} \subset X $ and a double sequence $(x_{n,m})_{n, m \in \NN} \subset X $ such that for each $n\in\NN$ we have $x_{n,m} \goesru x_n$ as $m \rightarrow \infty$ with respect to some regulator $u_n$.
Since $X$ satisfies condition $(R)$, there exists a sequence of positive scalars $(\lambda_{n})_{n \in \NN} $ and $u\in X_+$ such that $\lambda_{n}u_{n}\leq u$ for each $n\in\NN$. Since $x_{n,m}\goesru x_n$ as $m\to\infty$ is regulated by $\lambda_nu_n$, it is also regulated by $u$.
\end{proof}

Property (C) also holds for nets.
The following corollary follows directly from \Cref{property (C) iff regular}.

\begin{cor}\label{property (C) sequence to nets}
A vector lattice $X$ has \emph{property $(C)$} if and only if any countable set of relatively uniformly convergent nets in $X$ has a common regulator.
\end{cor}

It is easy to see that every vector lattice with an order unit has property $(C)$.
\Cref{implies property $(D)$} shows that the class of vector lattice  which have property $(D)$ is quite big. Apart to vector lattices with order units, it also contains the class of completely metrizable locally solid vector lattices.

\begin{thm}\label{implies property $(D)$}
For a vector lattice $X$ consider the following assertions.
\begin{enumerate}
\item [\upshape{(i)}] There exists a topology $\tau$ on $X$ such that $(X,\tau)$ is completely metrizable locally solid vector lattice.
\item [\upshape{(ii)}] $X$ has property $(C)$.
\item [\upshape{(iii)}] $X$ has property $(D)$.
\end{enumerate}
Then
\upshape{$$\text{(i)} \Rightarrow \text{(ii)} \Rightarrow \text{(iii)}.$$}
\end{thm}

\begin{proof}
(i)$\Rightarrow$(ii) By \Cref{property (C) iff regular}, it is enough to show that $X$ satisfies condition $(R)$. Since $(X,\tau)$ is completely metrizable there exists a countable neighborhood basis $\{ V_n\}_{n \in \NN}$ of zero in $(X,\tau)$ consisting of solid sets such that for each $n \in \NN$ we have $V_{n+1} + V_{n+1} \subset V_n$. Fix a sequence $(u_{n})_{n \in \NN} \subset X_+ $ and for each $n \in \NN$ pick $\lambda_n$ such that $\lambda_n u_n \in V_n$. We claim that the series $\sum_{i=1}^{\infty} \lambda_i u_{i}$ $\tau$-converges in $X$. Define $ s_n =\sum_{i=1}^{n} \lambda_i u_{i}$ for each $n \in \NN$ and pick a solid neighborhood $V_0$ of zero in $(X,\tau)$. Find $n_0 \in \NN$ such that $V_{n_0} \subset V_0$. Then for $m > n \geq n_0$ we have
$$s_m - s_n = \lambda_{n+1} u_{n+1} + \dots + \lambda_{m} u_{m} \in V_n \subset V_{n_0} \subset V_0,$$
and hence, the partial sums $(s_n)_{n \in \NN}$ of the series $\sum_{i=1}^{\infty} \lambda_i u_{i}$ form a Cauchy sequence in $(X,\tau)$. Since $(X,\tau)$ is complete and Hausdorff, by \cite[Theorem 2.21]{Aliprantis:03}, the series $\sum_{i=1}^{\infty} \lambda_i u_{i}$ converges in $X$ to some positive vector $u$. Now it is clear that for each $n\in\NN$ we have $\lambda_n u_{n} \leq u $ and so  $X$ satisfies condition $(R)$.

(ii)$\Rightarrow$(iii) Suppose that a net of linear operators $(T_\alpha)_\alpha$ on $X$ and an ru-dense subset $D\subset X$ satisfy  (a) and (b) from \Cref{property $(D)$ def}. We need to prove that $T_\alpha x\goesru 0$ for each $x\in X$.

Pick $x\in X$ and find $(x_n)_{n \in \NN} \subset D$ such that $x_n \xrightarrow{ru} x$. By (b), there exists
$u \in X_+$ such that for each $ \varepsilon >0$ there exists $N_\varepsilon \in \NN$ and $\alpha_\varepsilon$ such that $$|T_\alpha (x_{n}-x)| \leq \varepsilon \cdot u$$
holds for all $n \geq N_\varepsilon$ and $\alpha \geq\alpha_\varepsilon $. Since for each $n\in\NN$ the net $(T_\alpha x_{n})_{\alpha}$ converges relatively uniformly to $0$ and since $X$ has property $(C)$, by \Cref{property (C) sequence to nets}, there exists a positive vector $\widetilde u\in X$ which regulates  the convergence $T_\alpha x_n \xrightarrow{ru} 0$ for each $n\in\NN$. Now, find $\alpha_1$ such that for all $\alpha\geq \alpha_1$ we have $|T_\alpha x_{N_\varepsilon}|\leq \varepsilon \cdot \widetilde u.$ Find any $\alpha_0$ which is greater than or equal to $\alpha_\varepsilon$ and $\alpha_1.$ If $\alpha\geq \alpha_0$, then
\[|T_\alpha x|\leq |T_\alpha (x-x_{N_\varepsilon})|+|T_\alpha x_{N_\varepsilon}| \leq \varepsilon \cdot (u+ \widetilde u).\qedhere\]
\end{proof}

At first glance it may seem that we do not require property $(C)$ in its entirety. The problem is that by changing the value $\varepsilon$, also the integer $N_\varepsilon$ changes which unfortunately forces the vector $\widetilde u=\widetilde u(\varepsilon)$ to change itself. In this case, we cannot conclude that the vector $u+\widetilde u$ is a regulator of our convergence $T_\alpha x\goesru 0$.

The following two examples show that the implications from  \Cref{implies property $(D)$}, in general, cannot be reversed.

\begin{ex}
By \Cref{1}, the function $u\colon x \mapsto 1+|x|$ is an order unit of $\Lip(\RR)$ and hence, $\Lip(\RR)$ has  property $(C)$. If there would exist a complete metrizable locally solid topology $\tau$ on $\Lip(\RR)$, then $\tau=\tau_{ru}$ by \Cref{topology examples}.
Since $u$ is an order unit for $\Lip(\RR)$, $\tau_{ru}$ agrees with norm topology induced by $\|\cdot\|_u$. In order to reach a contradiction, we will show that the normed space $(\Lip(\RR),\|\cdot\|_u)$ is not complete.

 It is clear that each function $f_n\colon \RR\to \RR$ defined as $f_n(x)=\sqrt{|x|+\frac{1}{n}}$ is in $\Lip(\RR)$.
 A direct calculation shows that the sequence $(f_n)_{n\in \mathbb N}$ is Cauchy in $(\Lip(\RR),\|\cdot\|_u)$ and $f_n\xrightarrow{\|\cdot\|_u}f$ where $f(x)=\sqrt{|x|}.$ Since $f\notin \Lip(\RR)$, we conclude $(\Lip(\RR),\|\cdot\|_u)$ is not complete.
\end{ex}

\begin{ex}
The vector lattice $\C_c(\RR)$ has the
property $(D)$, yet it does not have property $(C)$.

To show that $\C_c(\RR)$ does not have property $(C)$, it suffices to check that $\C_c(\RR)$ does not satisfy the equivalent condition $(R)$. Pick any sequence $(f_n)_{n \in \NN} \subset \C_c(\RR)$ with $f_n=1$ on $[-n,n]$. It is easy to check that for any choice of positive scalars  $(\lambda_n)_{n \in \NN} \subset \RR$ the sequence $(\lambda_n f_n)_{n \in \NN}$ is not order bounded in $\C_c(\RR)$.

To show that $\C_c(\RR)$ has property $(D)$, suppose that (i) and (ii) from \Cref{property $(D)$ def} are satisfied for some ru-dense subset $D\subset \C_c(\RR)$ and some net $(T_\alpha)_\alpha$ of linear operators on $\C_c(\RR)$. We need to prove that $T_\alpha f\goesru 0$ for each $f\in \C_c(\RR)$.

Pick $f \in \C_c(\RR)$ and find $(f_n)_{n \in \NN} \subset D$ such that $f_n \goesru f$.
Pick arbitrary $\varepsilon>0$. First, by applying (ii), we find $N_\varepsilon \in \NN$ and $\alpha_\varepsilon$ such that $$|T_\alpha (f-f_{n})| \leq \frac{\varepsilon}{2\|u\|_\infty} \cdot u$$ holds for all $\alpha \geq \alpha_\varepsilon$ and $n\geq N_\varepsilon$. Next, by (i), there exist $\widetilde u\in \C_c(\RR)$ and $\alpha_1$ such that
$$|T_\alpha f_{N_\varepsilon}| \leq \frac{\varepsilon}{2\|\widetilde u\|_\infty} \cdot \widetilde u $$ holds for all $\alpha \geq \alpha_1$.
It is tempting to proceed with the same argument as in the proof of  (b)$\Rightarrow$(c) of \Cref{implies property $(D)$}, however since $\C_c(\RR)$ does not have property $(C)$ each choice of $\varepsilon$ provides a possibly different $N_\varepsilon$, and therefore a possibly different $\widetilde u.$

By \Cref{compact support ru}, it suffices to show that $T_\alpha f\xrightarrow{\|\cdot\|_\infty}0$ and that there exists a compact set $K$ and an index $\beta$ such that for each $\alpha\geq \beta$ the function $T_\alpha f$ vanishes  outside $K$. To see this, pick any $\alpha_0 \geq \alpha_\varepsilon, \alpha_1$ and observe that
$$|T_\alpha f|\leq |T_\alpha (f-f_{N_\varepsilon})|+|T_\alpha f_{N_\varepsilon}|\leq \frac{\varepsilon}{2\|u\|_\infty}\cdot u+\frac{\varepsilon}{2\|\widetilde u\|_\infty} \cdot \widetilde u$$
holds for all $\alpha \geq \alpha_0$. This yields $T_\alpha f\xrightarrow{\|\cdot\|_\infty}0$
and for each $\alpha\geq \alpha_0$ the function $T_\alpha f$ vanishes outside the union of supports of $u$ and $\widetilde u$. This finally proves the claim.
\end{ex}

In conclusion, the class of vector lattices which have property $(D)$ contains at least vector lattices such as $L^p(\RR)$ $(0<p<\infty)$, $\C_c(\RR)$, $\Lip(\RR)$, $\UC(\RR)$, and $\C(\RR)$ which are examples of very important spaces where one wants to solve (partial) differential equations.

The following theorem is the main result of this section. It is a version of \Cref{equicont to strongly continuous} for relatively uniformly continuous semigroups and it will be applied in the following section.
Furthermore, this result is of fundamental importance to the proof of \cite[Theorem 5.4]{Kaplin:18} which is a Hille-Yosida type result for relatively uniformly continuous semigroups.
\begin{thm}\label{ruc on dense to ruc on whole space}
Let $X$ have property $(D)$ and $(T(t))_{t\geq 0}$ be a positive semigroup on $X$. Then $(T(t))_{t\geq 0}$ is relatively uniformly continuous on $X$ if and only if the following two assertions hold.
\begin{enumerate}
\item [\textup{(i)}] There exists an $ru$-dense subset $D \subset X$ such that $T(h)y \goesru y$ as $h \searrow 0$ for each $y \in D$.
\item [\textup{(ii)}] For each $ s \geq 0$ and $x \in X$ the set $$\{|T(t)x| \ \colon \; 0\leq t\leq s\}$$ is order bounded in $X$.
\end{enumerate}
\end{thm}

\begin{proof}
($\Rightarrow$) If $(T(t))_{t\geq 0}$ is a positive relatively uniformly continuous semigroup on $X$, then (i) obviously holds for $D=X$ and (ii) follows from \Cref{uniform bound}.

$(\Leftarrow$) Fix $x \in X$ and define a net of linear operators $(T_h)_{h \in [0, 1]}$ on $X$ by $T_h:=T(h)-I$.
By \Cref{relatively uniformly continuous semigroup 0}, it suffices to prove that $T_hx\goesru 0$ as $h\searrow 0.$ Since $X$ has property $(D)$, it suffices to check (a) and (b) of \Cref{property $(D)$ def}.

Clearly (a) holds by assumption (i).  To check (b), fix a sequence $(x_n)_{n \in \NN}$ such that $x_n\goesru 0$ with respect to some regulator $u \in X_+$. By assumption (ii), we can find $ v \in X_+$ such that
$T(h)u \leq v$
for all $h \in [0, 1]$. Hence, for each $\varepsilon >0$ there exists $N_\varepsilon \in \NN$ such that for all $n\geq N_\epsilon$ and $h\in [0,1]$ we have
\[|T_h x_n| \leq T(h)| x_n|+ |x_n| \leq \varepsilon \cdot T(h)u + \varepsilon \cdot u \leq \varepsilon \cdot (v + u ).\qedhere \]
\end{proof}

\section{Koopman semigroups on $\C(\RR)$}\label{Koopman}
In this section we define continuous semiflows and the corresponding Koopman semigroups on $\C(\RR)$. By applying \Cref{ruc on dense to ruc on whole space}, we show that relative uniform continuity of such semigroups on $\Lip(\RR)$, $\UC(\RR)$ and $\C(\RR)$ can be characterized through continuity properties of their semiflows. In particular, we will show that the (left) translation semigroup is relatively uniformly continuous on these spaces. 

A function $\varphi  \colon \RR_+ \times \RR \rightarrow \RR$ is called a continuous \emph{semiflow} if it is continuous in the second variable and satisfies
\begin{equation}\label{semiflows}
\varphi (0 ,x)=x \qquad \text{and} \qquad \varphi (t+s,x)=\varphi(t, \varphi(s,x))
\end{equation}
for each $x \in \RR$ and $t,s \geq 0$. 
To each semiflow $\varphi$ we associate the family of operators $\mathcal{T}_\varphi = (T_\varphi(t))_{t\geq 0}$ on $\C(\RR)$ given by
$$(T_\varphi(t)f)(x)=f(\varphi (t,x))$$
for each $x \in \RR$ and $t \geq 0$. We call such a semigroup a \emph{Koopman semigroup}.
For example, the semigroup associated with the semiflow $(t,x) \mapsto t+x$ is the (left) translation semigroup. 

The following lemma, which is of independent interest, will be used to prove \Cref{3}.  For the sake of completeness we include the proof.
\begin{lemma}\label{thy lemma}
Let $\varphi  \colon \RR_+ \times \RR \rightarrow \RR$ be a semiflow and assume that there exists $u \in \C(\RR)$ such that for each $\varepsilon >0$ there exists $\delta >0$ such that
\begin{equation}\label{hjhjhjhjj}
|\varphi(h,z)-z| \leq \varepsilon \cdot u(z)
\end{equation}
holds for all $h \in [0, \delta]$ and $z \in \RR$. Then $\varphi$ is jointly continuous and  for each $f \in \C(\RR)$ and $s \geq 0$ the function $x \mapsto g_{f,s}(x):= \max_{t \in [0,s]}|f(\varphi(t,x))|$ is continuous.
\end{lemma}
\begin{proof}
We first show that $\varphi$ is jointly continuous. Fix $ t > 0$, $x \in \RR$ and pick $\varepsilon >0$. Due to the fact that the function $y \mapsto \varphi (t,y)$ is continuous, there exists $0< \delta <\min\{1,t\}$ such that for each $y \in [- \delta, \delta]$ we have $$| \varphi(t,y+x)-\varphi(t,x)| \leq \frac{\varepsilon}{2}. $$
Also, since the function $y \mapsto u(\varphi (t,y+x))$ is continuous $M:=  \max_{y \in [0,1]}u(\varphi(t,y+x))$ is well-defined.
 By (\ref{hjhjhjhjj}), there exists $0<\sigma < \delta$ such that for all $h \in [- \sigma, \sigma]$ and $y \in \RR$ we obtain
$$|\varphi(h,\varphi(t,y+x))-\varphi(t,y+x)|= |\varphi(h,z_y)-z_y|\leq \frac{\varepsilon}{2  M} \cdot u(z_y) =\frac{\varepsilon}{2  M} \cdot u(\varphi(t,y+x)). $$
where $z_y =\varphi(t,y+x)$. Hence, by using (\ref{semiflows}), we conclude
\begin{align*}
|\varphi(h +t,y+x)-\varphi(t,x)| &\leq|\varphi(h,\varphi(t,y+x))-\varphi(t,y+x)|+  |\varphi(t,y+x)-\varphi(t,x)| \\
& \leq \frac{\varepsilon}{2 M} \cdot u(\varphi(t,y+x)) +  \frac{\varepsilon}{2} \leq \varepsilon
\end{align*}
for all $h,y \in [- \sigma, \sigma]$ which proves that $\varphi$ is jointly continuous.

Next, pick  $f \in \C(\RR)$, $s \geq 0$, $x \in \RR$ and $(x_n)_{n \in \NN} \subset \RR$  such that $x_n \rightarrow x$ as $n \rightarrow \infty$. If we can show that each subsequence $(y_n)_{n \in \NN}$ of $(x_n)_{n \in \NN}$ has a subsequence $(y_{n_k})_{n \in \NN}$ such that $g_{f,s}(y_{n_k}) \rightarrow g_{f,s}(x)$ as $k \rightarrow \infty$, then we conclude $g_{f,s}(x_n) \rightarrow g_{f,s}(x)$ as $k \rightarrow \infty$ from where it follows that $g_{f,s}$ is continuous.  To see this, we first notice that $(\tau,y) \mapsto |f(\varphi(\tau,y))|$ is jointly continuous. Hence, there exist $t, t_n \in [0,s]$ such that for each $n \in \NN$ we have \begin{equation}\label{nnnnn}
g_{f,s}(y_n) = |f(\varphi(t_n,y_n))| \quad \text{and} \quad g_{f,s}(x) =|f(\varphi(t,x))|.
\end{equation} 
We choose a converging subsequence $(t_{n_k})_{k \in \NN} $ with limit $t^* \in [0,s]$ and observe that $|f(\varphi(t_{n_k},y_{n_k}))| \rightarrow |f(\varphi(t^*,x))|$ as $k \rightarrow \infty$.  By this observation and (\ref{nnnnn}), it is enough to show that $g_{f,s}(x) = |f(\varphi(t^*,x))|$ to conclude that $g_{f,s}(y_{n_k}) \rightarrow g_{f,s}(x)$ as $k \rightarrow \infty$. First, by definition, we have $g_{f,s}(x)  \geq |f(\varphi(t^*,x))|$. To see that $g_{f,s}(x)  \leq |f(\varphi(t^*,x))|$, fix $\varepsilon>0$ and note that, by joint continuity of $(\tau,y) \mapsto f(\varphi(\tau,y))$ and $(\tau,y) \mapsto |f(\varphi(\tau,y))|$, there exists $K \in \NN$ such that $$|f(\varphi(t,x))- f(\varphi(t,y_{n_K}))| \leq \frac{\varepsilon}{2} \quad \text{and} \quad |f(\varphi(t_{n_K },y_{n_K})|- |f(\varphi(t^*,x))| \leq \frac{\varepsilon}{2} $$ hold, respectively. Hence, by applying (\ref{nnnnn}), we estimate \begin{align*}
|g_{f,s}(x)| & = |f(\varphi(t,x))| \leq |f(\varphi(t,x))- f(\varphi(t,y_{n_K}))| +  |f(\varphi(t,y_{n_K}))| \leq \frac{\varepsilon}{2} + g_{f,s}(y_{n_K}) \\
& \leq \frac{\varepsilon}{2} +|g_{f,s}(y_{n_K})- f(\varphi(t^*,x))| +  |f(\varphi(t^*,x))|\leq    \varepsilon + |f(\varphi(t^*,x))|.
\end{align*}
Since $\varepsilon$ was chosen arbitrarily we conclude $g_{f,s}(x)  \leq |f(\varphi(t^*,x))|$. 
\end{proof}


In order to characterize relatively uniformly continuous Koopman semigroups on $\C(\RR)$ we need to consider the space $\mathcal{L}\mathcal{P}\mathcal{A}(\RR)$
of \emph{locally piecewise affine functions on $\RR$}. A function  $f\in \C(\RR)$ is locally piecewise affine if there exist sequences $(a_n)_{n \in \ZZ}, (b_n)_{n \in \ZZ},(j_n)_{n \in \ZZ}\subset \RR$
such that $\bigcup_{n \in \ZZ}[j_n, j_{n+1}] = \RR$ and 
\begin{enumerate}
\item [(1)] $j_n<j_{n+1}$,
\item [(2)] $f(x)=a_n \cdot x + b_n$ holds for all $x \in [j_n, j_{n+1}] $, and
\item [(3)] $b_{n-1}-b_n=(a_n-a_{n-1})j_n$
\end{enumerate}
hold for all $n \in \ZZ$.
It is easy to see that our definition is equivalent to the definition of $\mathcal{L}\mathcal{P}\mathcal{A}(\RR^N)$ from  \cite{Adeeb:17} in the case when $N=1$. \cite[Theorem 4.1]{Adeeb:17} yields that $\mathcal L \mathcal P\mathcal A(\RR)$ is ru-dense in $\C(\RR)$.

\begin{prop}\label{3}
For a semiflow $\varphi$ the following assertions are equivalent.
\begin{enumerate}
\item [\textup{(i)}] $\mathcal{T}_\varphi$ is relatively uniformly continuous on $\C(\RR)$.
\item [\textup{(ii)}] There exists $u \in \C(\RR)$ such that for each $\varepsilon >0$ there exists $\delta >0$ such that
$$|\varphi(t,x)-x| \leq \varepsilon \cdot u(x)$$
holds for all $t \in [0, \delta]$ and $x \in \RR$.
\item [\textup{(iii)}] For each $f \in \mathcal{L}\mathcal{P}\mathcal{A}(\RR)$ there exists $v \in \mathcal{LPA}(\RR)$ such that $T_\varphi(h)f \goesru f$ with respect to $v$ as $h \searrow 0$.
\end{enumerate}
\end{prop}

\begin{proof}
(i)$\Rightarrow$(ii)  By assumption, there exists $u \in X_+$ such that for each $\varepsilon >0$ there exists $\delta >0$ such that
$$|\varphi(t,\cdot)-\Id_\RR| =|T_\varphi(t)\Id_\RR-\Id_\RR| \leq \varepsilon \cdot u$$
holds for all $t \in [0, \delta]$.

(ii)$\Rightarrow$(iii) Fix $f \in \mathcal{L}\mathcal{P}\mathcal{A}(\RR)$ and pick sequences $(a_n)_{n \in \ZZ}, (b_n)_{n \in \ZZ}, (j_n)_{n \in \ZZ}\subset \RR$ from  the characterization of functions in $\mathcal{L}\mathcal{P}\mathcal{A}(\RR)$. Our goal is to construct a function $v \in \mathcal{LPA}(\RR)$ such that it regulates $T_\varphi(h)f\goesru f$ as $h\searrow 0$. To this end, pick $n\in\ZZ$, set $$\delta_n:=\min \left \{\frac{j_{i+1}-j_{i}}{2}: n-1 \leq i \leq n+1 \right \}$$
 and $M_n:=\max_{x \in [j_n, j_{n+1}]} u(x)$. By assumption, there exists $s_n >0$ such that \begin{equation}\label{464646}
 |\varphi (t,x)-x| \leq \frac{\delta_n}{\max\{M_n,1\}} \cdot u(x) \leq \delta_n
 \end{equation}
holds for all $x \in [j_n, j_{n+1}]$ and $t \in [0,s_n]$. Set
\begin{align*}
c_n&:= {s_n}^{-1} \cdot\sup_{(t,x) \in [0,1] \times  [j_n,j_{n+1}]}|f(\varphi (s_n+t,x))-f(\varphi (s_n,x))|, \\
d_n&:= \max\{M_n,1\} \cdot \max_{n-1 \leq i \leq n +1} \left\{|a_{i}|, |a_{i}-a_n|, c_i \right\} 
\end{align*}
and finally define the function $v\colon \mathbb R\to \mathbb R$ on each interval $[j_n,j_{n+1}]$ separately; if $x\in [j_n,j_{n+1}]$ we define
$$v(x)=(d_{n+1}-d_{n}) \frac{x-j_{n}}{j_{n+1}-j_{n}} + d_n.$$
A direct verification shows that $v \in \mathcal{LPA}(\RR).$ We claim that the function $v$ regulates $T_\varphi(h)f\goesru f$ as $h\searrow 0.$ To this end, choose $0<\varepsilon<1$. By assumption, there exists $0< \delta < \dfrac{\varepsilon}{2}$ such that $$|\varphi(h,x)-x| \leq \dfrac{\varepsilon}{4} \cdot u(x)$$ holds for all $h \in (0,\delta]$ and $x \in \RR$. We will prove that for each  $n\in\ZZ$, $x \in [j_n,j_{n+1}]$ and $h \in (0, \delta]$ we have $|f(\varphi(h,x))-f(x)|\leq \varepsilon \cdot v(x)$.

\underline{Case 1}: Assume that $|\varphi(h,x)-x|<\delta_n$. Then $\varphi(h,x) \in [j_i,j_{i+1}]$ for some $n-1 \leq i \leq n+1$ and hence, we estimate
\begin{align*}
|f(\varphi(h,x))-f(x)|&=|a_i\varphi(h,x) + b_i- (a_n x + b_n)| \leq |\varphi(h,x)-x| \cdot |a_i| + |(a_i-a_n) x- (b_n-b_i)|\\
& \leq \dfrac{\varepsilon}{4} \cdot   u(x) \cdot |a_i| + |a_i-a_n| \cdot | x- j_i| .
\end{align*}
Hence, if $\varphi(h,x) \in [j_n,j_{n+1}]$, i.e. $i=n$, then we obtain $$|f(\varphi(h,x))-f(x)| \leq \dfrac{\varepsilon}{4} \cdot   u(x) \cdot |a_n| \leq \dfrac{\varepsilon}{4} \cdot M_n\cdot |a_n| \leq \dfrac{\varepsilon}{4} \cdot \min\{ d_n, d_{n+1}\} \leq \dfrac{\varepsilon}{4} \cdot v(x).$$
Furthermore, if $\varphi(h,x) \in [j_{n+1},j_{n+2}]$, i.e. $i=n+1$,  then $ j_{n+1}-x < \varphi(h,x)-x \leq \dfrac{\varepsilon}{4} \cdot   u(x) $, from where we conclude 
\begin{align*}
|f(\varphi(h,x))-f(x)| &\leq  \dfrac{\varepsilon}{4} \cdot   u(x) \cdot |a_{n+1}| + | a_{n+1}- a_{n}| \cdot  (j_{n+1}-x) \\
&\leq \dfrac{\varepsilon}{4} \cdot M_n\cdot (|a_{n+1}|+ | a_{n+1}- a_{n}|) 
 \leq \dfrac{\varepsilon}{2} \cdot \min\{ d_n, d_{n+1}\} \leq \dfrac{\varepsilon}{2} \cdot v(x).
\end{align*}
Similarly, we argue when $\varphi(h,x) \in [j_{n-1},j_{n}]$. 

\underline{Case 2}: Assume that $|\varphi(h,x)-x| \geq \delta_n$. By (\ref{464646}), there exists $0< s_n <h$ such that $|\varphi(t,x)-x| \leq \delta_n$ holds for all $t \in [0,s_n]$. We write $h= N_n s_n +r_n$ for some $N_n \in \NN$ and $0 \leq r_n < s_n$. Case 1 yields $|f(\varphi(r_n,x))-f(x)| \leq \dfrac{  \varepsilon}{2} \cdot v(x)$. By an easy application of the triangle inequality we can estimate
$$|f(\varphi(h,x))-f(x)| \leq \sum_{m=1}^{N_n}|f(\varphi(m s_n +r_n,x))-f(\varphi((m-1) s_n +r_n,x))|+ |f(\varphi(r_n,x))-f(x)|.$$
For each $1\leq m\leq N_n$ we denote $t_m:=(m-1) s_n +r_n$. Since  $\varepsilon<1$, also $h<1$ and hence, $ y_m \in [0,1]$. By definition of the number $c_n$ and the estimate $c_n \leq \min\{d_n, d_{n+1}\} \leq v(x)$, we obtain
\begin{align*}
|f(\varphi(h,x))-f(x)|&\leq \sum_{m=1}^{N_n}|f(\varphi(s_n +t_m,x))-f(\varphi(t_m,x))|+|f(\varphi(r_n,x))-f(x)|\\
&\leq N_n  s_n \cdot c_n+ \dfrac{  \varepsilon}{2} \cdot v(x)\leq \varepsilon\cdot v(x).
\end{align*}

(iii)$\Rightarrow$(i)  The vector lattice $\C(\RR)$ equipped with the topology of uniform convergence on compact sets is a completely metrizable locally solid vector lattice and hence, by \Cref{implies property $(D)$}, it has property $(D)$. Hence, by \Cref{ruc on dense to ruc on whole space}, it is enough to check  the assertions that there exists an $ru$-dense subset $D \subset \C(\RR)$ such that $T(h)g \goesru g$ as $h \searrow 0$ for each $g \in D$ and that for each $ s \geq 0$ and $f \in \C(\RR)$ there exists $u \in \C(\RR)$ such that  $|T(t)f| \leq u$ holds for all $t \in [0,s]$. The first assertion follows directly from the  assumption and  \cite[Theorem 4.1]{Adeeb:17} which yields that $\mathcal L \mathcal P\mathcal A(\RR)$ is ru-dense in $\C(\RR)$. Pick $f \in \C(\RR)$ and $s \geq 0$.  Using the same argument as in (i)$ \Rightarrow$(ii), we see that the semiflow $\varphi$ satisfies (ii) and hence, by \Cref{thy lemma}, the function $x \mapsto g_{f,s}(x):= \max_{t \in [0,s]}|f(\varphi(t,x))|$ is continuous. We conclude the proof by noting that $|T_\varphi(t)f| \leq g_{f,s}$ holds for all $t \in [0,s]$. 
\end{proof}

The following proposition characterizes relatively uniformly continuous Koopman semigroups on $\Lip(\RR)$ and $\UC(\RR)$ through their semiflows.

\begin{prop}\label{2}
Let $X=  \Lip(\RR)$ or $X=\UC(\RR)$ and $\varphi$ be a semiflow such that the operators of the (semi)group $\mathcal{T}_\varphi$ leave $X$ invariant. The following assertions are equivalent.
\begin{enumerate}
\item [\textup{(i)}] $\mathcal{T}_\varphi$ is relatively uniformly continuous on $X$.
\item [\textup{(ii)}] There exists $u \in X$ such that for each $\varepsilon >0$ there exists $\delta >0$ such that
$$|\varphi(h,x)-x| \leq \varepsilon \cdot u(x)$$
holds for all $h \in [0, \delta]$ and $x \in \RR$.
\item [\textup{(iii)}] For each $\varepsilon >0$ there exists $\delta >0$ such that
$$|\varphi(h,x)-x| \leq \varepsilon \cdot (1+|x|)$$
holds for all $h \in [0, \delta]$ and $x \in \RR$.
\end{enumerate}
\end{prop}
\begin{proof}
(i)$\Rightarrow$(ii) follows from the same argument as in the proof of (i)$\Rightarrow$(ii) of \Cref{3}.

(ii)$\Rightarrow$(iii) follows from the fact that the function $x \mapsto 1+|x|$ is an order unit in $X$.

(iii)$\Rightarrow$(i) Fix $f \in X$ and $ \varepsilon>0$. By \Cref{1}, we can find some $\delta >0$ such that $$|f(x)-f(y)| \leq \varepsilon \cdot (|x-y|\cdot \delta^{-1}+1)$$ for all $x,y \in \RR$. By assumption, there exists $\sigma >0$ such that $|\varphi(h,x)-x| \leq \delta \cdot (1+|x|)$ holds for all $h \in [0, \sigma]$ and $x \in \RR$. Now it easily follows that
\begin{align*}\label{ucf}
|(T_\varphi(h)f)(x)-f(x)| &=|f(\varphi(h,x))-f(x)| \leq \varepsilon \cdot (|\varphi(h,x)-x| \cdot \delta^{-1} +1)  \leq \varepsilon \cdot (2+|x|)
\end{align*}
holds for all $h \in [0, \sigma]$ and $x \in \RR$. We conclude the proof by noting that  $x \mapsto 2+|x|$ is in $X$ and by applying \Cref{relatively uniformly continuous semigroup 0}.
\end{proof}
\begin{rem}For $X=  \Lip(\RR)$ or $X=\UC(\RR)$ and a semiflow $\varphi$ one can easily verify that the operators of $\mathcal{T}_\varphi$ leave $X$ invariant if and only if  for each $t \geq 0$ the mapping $x \mapsto \varphi(t,x)$ is in $X$.
\end{rem}
For the semiflow $(t,x) \mapsto t+x$, Propositions \ref{translation semigroup on Cc}, \ref{3}, and \ref{2} yield that the (left) shift semigroup is relatively uniformly continuous on $\C_c(\RR)$, $\C(\RR)$, $\Lip(\RR)$, and $\UC(\RR)$, respectively.


\section*{Acknowledgments and further remarks}
The results of this paper are used in \cite{Kaplin:18} which is the continuation of our investigation of relatively uniformly continuous semigroups.

The authors would like to thank Marjeta Kramar Fijav\v z for fruitful discussions and comments that improved the manuscript. The second author would also like to thank Marjeta for scientific guidance.

\bibliographystyle{alpha}
\bibliography{literature}

\newcommand{\etalchar}[1]{$^{#1}$}
\begin{thebibliography}{AGG{\etalchar{+}}86}

\bibitem[AB03]{Aliprantis:03}
C.~D. Aliprantis and O.~Burkinshaw.
\newblock {\em Locally solid {R}iesz spaces with applications to economics},
  volume 105 of {\em Mathematical Surveys and Monographs}.
\newblock American Mathematical Society, Providence, RI, second edition, 2003.

\bibitem[AGG{\etalchar{+}}86]{Arendt:86}
W.~Arendt, A.~Grabosch, G.~Greiner, U.~Groh, H.~P. Lotz, U.~Moustakas,
  R.~Nagel, F.~Neubrander, and U.~Schlotterbeck.
\newblock {\em One-parameter semigroups of positive operators}, volume 1184 of
  {\em Lecture Notes in Mathematics}.
\newblock Springer-Verlag, Berlin, 1986.

\bibitem[Amb42]{Ambrose:42}
Shizuo Ambrose, Warren;~Kakutani.
\newblock Structure and continuity of measurable flows.
\newblock {\em Duke Mathematical Journal}, 9, 03 1942.

\bibitem[AT07]{Aliprantis:07}
Charalambos~D. Aliprantis and Rabee Tourky.
\newblock {\em Cones and duality}, volume~84 of {\em Graduate Studies in
  Mathematics}.
\newblock American Mathematical Society, Providence, RI, 2007.

\bibitem[AT17]{Adeeb:17}
S.~Adeeb and V.~G. Troitsky.
\newblock Locally piecewise affine functions and their order structure.
\newblock {\em Positivity}, 21(1):213--221, 2017.

\bibitem[BKR17]{Batkai:17}
Andr\'as B\'atkai, Marjeta {Kramar~Fijav\v{z}}, and Abdelaziz Rhandi.
\newblock {\em Positive operator semigroups: From finite to infinite
  dimensions}, volume 257 of {\em Operator Theory: Advances and Applications}.
\newblock Birkh\"auser/Springer, Cham, 2017.

\bibitem[BMM12]{Budisic:12}
Marko Budišić, Ryan Mohr, and Igor Mezic.
\newblock Applied koopmanism.
\newblock {\em Chaos (Woodbury, N.Y.)}, 22:047510, 12 2012.

\bibitem[BR84]{Batty:84}
Charles J.~K. Batty and Derek~W. Robinson.
\newblock Positive one-parameter semigroups on ordered {B}anach spaces.
\newblock {\em Acta Appl. Math.}, 2(3-4):221--296, 1984.

\bibitem[EGK18]{Edeko:18}
Nikolai Edeko, Moritz Gerlach, and Viktoria Kühner.
\newblock Measure-preserving semiflows and one-parameter koopman semigroups.
\newblock {\em Semigroup Forum}, 2018.

\bibitem[EL15]{terElst:15}
A.~F. M.~ter Elst and M.~Lemanczyk.
\newblock On one-parameter koopman groups.
\newblock {\em Ergodic Theory and Dynamical Systems}, -1, 04 2015.

\bibitem[EN00]{Engel:00}
Klaus-Jochen Engel and Rainer Nagel.
\newblock {\em One-parameter semigroups for linear evolution equations}, volume
  194 of {\em Graduate Texts in Mathematics}.
\newblock Springer-Verlag, New York, 2000.

\bibitem[Hil42]{Hille:42}
Einar Hille.
\newblock Representation of one-parameter semigroups of linear transformations.
\newblock {\em Proc. Nat. Acad. Sci. U. S. A.}, 28:175--178, 1942.

\bibitem[Hil48]{Hille:48}
Einar Hille.
\newblock {\em Functional {A}nalysis and {S}emi-{G}roups}.
\newblock American Mathematical Society Colloquium Publications, vol. 31.
  American Mathematical Society, New York, 1948.

\bibitem[KK18]{Kaplin:18}
M.~Kaplin and M.~{Kramar~Fijav\v{z}}.
\newblock Generation of relatively uniformly continuous semigroups on vector
  lattices.
\newblock 2018.
\newblock arXiv:1410.5093.

\bibitem[KN42]{Koopman:42}
B.O. Koopman and J.~von Neumann.
\newblock Structure and continuity of measurable flows.
\newblock {\em Duke Mathematical Journal}, 9, 03 1942.

\bibitem[Kom64]{Komatsu:64}
Hikosaburo Komatsu.
\newblock Semi-groups of operators in locally convex spaces.
\newblock {\em J. Math. Soc. Japan}, 16:230--262, 1964.

\bibitem[Kom68]{Komura:68}
Takako Komura.
\newblock Semigroups of operators in locally convex spaces.
\newblock {\em J. Functional Analysis}, 2:258--296, 1968.

\bibitem[Koo31]{Koopman:31}
B.~O. Koopman.
\newblock Hamiltonian systems and transformation in hilbert space.
\newblock {\em Proceedings of the National Academy of Sciences}, 17, 05 1931.

\bibitem[LM67]{Luxemburg:67}
W.~A.~J. Luxemburg and L.~C. Moore, Jr.
\newblock Archimedean quotient {R}iesz spaces.
\newblock {\em Duke Math. J.}, 34:725--739, 1967.

\bibitem[LZ71]{Luxemburg:71}
W.~A.~J. Luxemburg and A.~C. Zaanen.
\newblock {\em Riesz spaces. {V}ol. {I}}.
\newblock North-Holland Publishing Co., Amsterdam-London; American Elsevier
  Publishing Co., New York, 1971.
\newblock North-Holland Mathematical Library.

\bibitem[MG16]{Mauroy:16}
Alexandre Mauroy and Jorge Goncalves.
\newblock Linear identification of nonlinear systems: A lifting technique based
  on the koopman operator.
\newblock {\em preprint}, 05 2016.
\newblock arXiv:1605.04457.

\bibitem[Miy59]{Miyadera:59}
Isao Miyadera.
\newblock Semi-groups of operators in {F}r\'echet space amd applications to
  partial differential equations.
\newblock {\em T\^ohoku Math. J. (2)}, 11:162--183, 1959.

\bibitem[Moo68]{Moore:68}
L.~C. Moore, Jr.
\newblock The relative uniform topology in {R}iesz spaces.
\newblock {\em Nederl. Akad. Wetensch. Proc. Ser. A 71 = Indag. Math.},
  30:442--447, 1968.

\bibitem[\=O73]{Ouchi:73}
Sunao \=Ouchi.
\newblock Semi-groups of operators in locally convex spaces.
\newblock {\em J. Math. Soc. Japan}, 25:265--276, 1973.

\bibitem[SV65]{Singbal:65}
K.~Singbal-Vedak.
\newblock A note on semigroups of operators on a locally convex space.
\newblock {\em Proc. Amer. Math. Soc.}, 16:696--702, 1965.

\bibitem[Swa88]{Swartz:88}
Charles Swartz.
\newblock The {B}anach-{S}teinhaus theorem for ordered spaces.
\newblock In {\em Generalized functions, convergence structures, and their
  applications ({D}ubrovnik, 1987)}, pages 425--432. Plenum, New York, 1988.

\bibitem[Vul67]{Vulikh:67}
B.~Z. Vulikh.
\newblock {\em Introduction to the theory of partially ordered spaces}.
\newblock Translated from the Russian by Leo F. Boron, with the editorial
  collaboration of Adriaan C. Zaanen and Kiyoshi Is\'eki. Wolters-Noordhoff
  Scientific Publications, Ltd., Groningen, 1967.

\bibitem[Yos48]{Yosida:48}
K\^osaku Yosida.
\newblock On the differentiability and the representation of one-parameter
  semi-group of linear operators.
\newblock {\em J. Math. Soc. Japan}, 1:15--21, 1948.

\bibitem[Yos65]{Yosida:65}
K\^osaku Yosida.
\newblock {\em Functional analysis}.
\newblock Die Grundlehren der Mathematischen Wissenschaften, Band 123. Academic
  Press, Inc., New York; Springer-Verlag, Berlin, 1965.

\bibitem[Zaa97]{Zaanen:97}
A.~C. Zaanen.
\newblock {\em Introduction to operator theory in {R}iesz spaces}.
\newblock Springer-Verlag, Berlin, 1997.

\end{thebibliography}

%
%

\end{document}